%% file: Ranking_GLM.tex
\documentclass[11pt]{article}
\usepackage[T1]{fontenc}
\usepackage[latin9]{inputenc}
\usepackage{geometry}
\usepackage{verbatim}
\usepackage{mathtools}
\usepackage{bm}
\usepackage{amsmath}
\usepackage{amssymb}
\usepackage{centernot}
\usepackage[unicode=true,
 bookmarks=false,
 breaklinks=false,pdfborder={0 0 1},colorlinks=false, hidelinks]
 {hyperref}
\hypersetup{colorlinks, linkcolor=blue,anchorcolor=blue,citecolor=blue, urlcolor=black}

\usepackage{booktabs}

\makeatletter

\providecommand{\tabularnewline}{\\}

\usepackage{mathtools}
\usepackage{multirow}
\usepackage{hhline}

\usepackage{cite}\usepackage{amsthm}\usepackage{dsfont}\usepackage{array}\usepackage{mathrsfs}\usepackage{comment}\onecolumn
\usepackage[round]{natbib}
\usepackage{color}

\newcommand{\R}{\mathbb{R}}

\allowdisplaybreaks

\usepackage{enumitem}
\setlist[itemize]{leftmargin=1.5em}
\setlist[enumerate]{leftmargin=1.5em}

\usepackage{algorithm}
\usepackage{algorithmic}
\usepackage{arydshln}

\newcommand{\ba}{\bm{a}}

\newcommand{\bz}{\bm{z}}

\newcommand{\bI}{\bm{I}}

\newcommand{\bX}{\bm{X}}

\DeclareMathOperator{\EE}{\mathbb{E}}

\newcommand{\btheta}{\bm{\theta}}

\newcommand{\bmu}{\bm{\mu}}

\DeclareMathOperator{\ind}{\mathds{1}}  

\newcommand{\mymid}{\,|\,}

\numberwithin{equation}{section}
\renewcommand{\star}{\ast}

\definecolor{wj}{RGB}{0,0,200}
\definecolor{yly}{RGB}{0,150,0}

\makeatother

\begin{document}
\theoremstyle{plain} \newtheorem{lemma}{\textbf{Lemma}} \newtheorem{prop}{\textbf{Proposition}}
\newtheorem{theorem}{\textbf{Theorem}}\setcounter{theorem}{0}
\newtheorem{corollary}{\textbf{Corollary}} 
\newtheorem{assumption}{\textbf{Assumption}}
\newtheorem{assumptions}{\textbf{Assumption}} 
\newtheorem{example}{\textbf{Example}} 
\newtheorem{definition}{\textbf{Definition}}
\newtheorem{fact}{\textbf{Fact}} \newtheorem{condition}{\textbf{Condition}}\theoremstyle{definition}

\theoremstyle{remark}\newtheorem{remark}{Remark}\newtheorem{claim}{{Fact}}\newtheorem{conjecture}{\textbf{Conjecture}}

\title{Isotonic Mechanism for Exponential Family Estimation in Machine Learning Peer Review}
\author{Yuling Yan\thanks{Department of Statistics, University of Wisconsin-Madison, Madison, WI 53706, USA; Email: \texttt{yuling.yan@wisc.edu}.}\and Weijie J.~Su\thanks{Department of Statistics and Data Science, The Wharton School, University
		of Pennsylvania, Philadelphia, PA 19104, USA; Email: \texttt{suw@wharton.upenn.edu}.} \and Jianqing Fan\thanks{Department of Operations Research and Financial Engineering, Princeton
		University, Princeton, NJ 08544, USA; Email: \texttt{jqfan@princeton.edu}.}}

\maketitle

\begin{abstract}
In 2023, the International Conference on Machine Learning (ICML) required authors with multiple submissions to rank their submissions based on perceived quality. In this paper, we aim to employ these author-specified rankings to enhance peer review in machine learning and artificial intelligence conferences by extending the Isotonic Mechanism to exponential family distributions. This mechanism generates adjusted scores that closely align with the original scores while adhering to author-specified rankings. Despite its applicability to a broad spectrum of exponential family distributions, implementing this mechanism does not require knowledge of the specific distribution form. We demonstrate that an author is incentivized to provide accurate rankings when her utility takes the form of a convex additive function of the adjusted review scores. For a certain subclass of exponential family distributions, we prove that the author reports truthfully \textit{only if} the question involves only pairwise comparisons between her submissions, thus indicating the optimality of ranking in truthful information elicitation. Moreover, we show that the adjusted scores improve dramatically the estimation accuracy compared to the original scores and achieve nearly minimax optimality when the ground-truth scores have bounded total variation. We conclude with a numerical analysis of the ICML 2023 ranking data, showing substantial estimation gains in approximating a proxy ground-truth quality of the papers using the Isotonic Mechanism.
\end{abstract}



\input{intro.tex} 
\input{truth.tex} 
\input{hyperplane}
\input{proof_minimax.tex}
\input{discussion.tex}

\section*{Acknowledgements}
J.~Fan was supported by ONR grant N00014-19-1-2120 and the NSF grants DMS-2053832,  DMS-2210833, and DMS-2412029.
W.J.~Su was supported in part by NSF grants CCF-1934876 and CAREER DMS-1847415. Y.~Yan was supported in part by the Charlotte Elizabeth Procter Honorific Fellowship from Princeton University and the Norbert Wiener Postdoctoral Fellowship from MIT. We would like to thank Bart Lipman for suggesting relevant references.

\bibliographystyle{plainnat}
\bibliography{bibfile}

\appendix

\input{appendix}


\end{document}

%% file: intro.tex
\section{Introduction}

In recent years, the surge in submissions to artificial intelligence and machine learning conferences, such as ICML, NeurIPS, and ICLR, has made it increasingly challenging to maintain the quality of peer review. For instance, the number of submissions to NeurIPS has grown from 1,673 in 2014 to a record-high of 10,411 in 2022~\citep{nipsexp, nips2020,accept}. This increase largely results from the fact that it has become common for an author to submit \textit{multiple} papers to a single artificial intelligence and machine learning conference. For example, one author submitted 32 papers to ICLR 2020 \citep{iclr}. However, the number of qualified reviewers and the time they can offer are limited. This discrepancy has compelled conference organizers to recruit many novice reviewers, which has consequently compromised the quality of peer review~\citep{stelmakh2021novice,shah2022challenges}.

Unlike review processes in most journals, reviewers for machine learning conferences are required to score papers based on quality (for example, on a scale from 1 to 10), and the average score across typically three reviewers serves as the single most important factor in deciding whether to accept or reject a paper. As a manifestation of the decline in review quality, it is increasingly common for review scores to exhibit a significant level of variability (see Figure~\ref{fig:iclr}). The arbitrariness of these scores makes it very difficult to make sound decisions regarding whether to accept or reject a paper, or to recommend paper awards. For example, the 2014 NIPS experiment showed that approximately 57\% of the accepted papers would be rejected if reviewed twice~\citep{nipsexp}.

\begin{figure}[t]
	\centering
	
	\begin{tabular}{c}
		\includegraphics[scale=0.45]{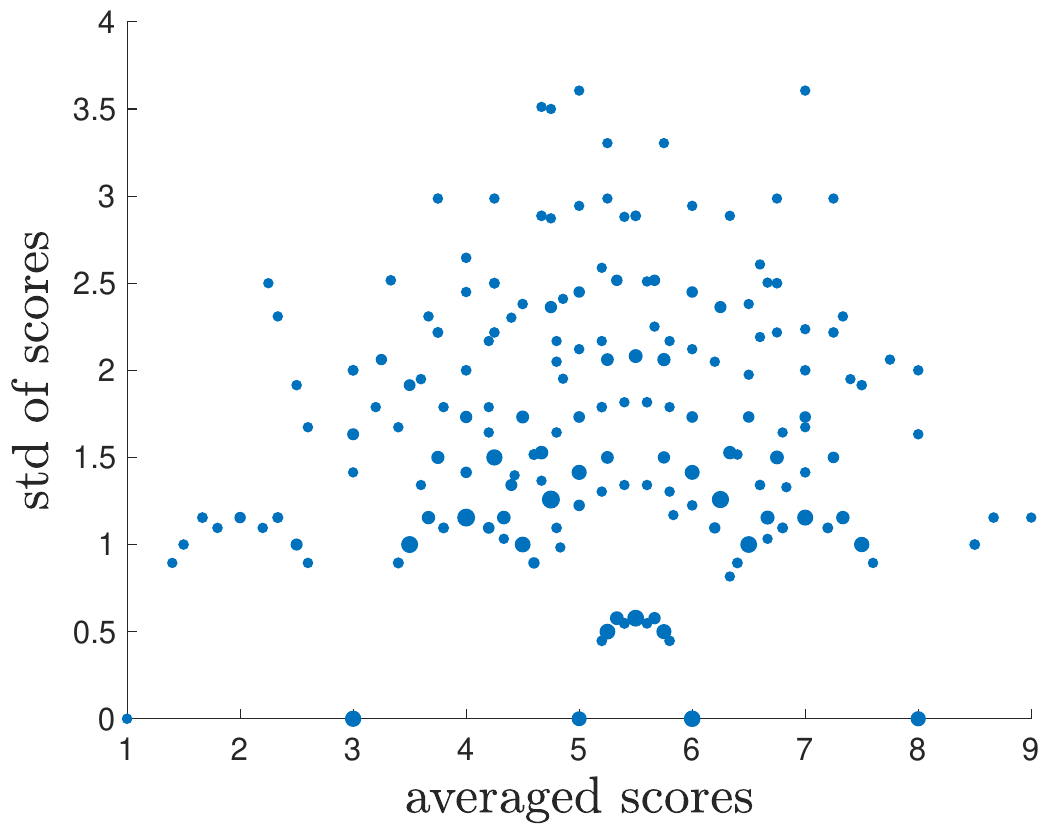}
	\end{tabular}
	
	\caption{Average review scores versus their standard deviations for all papers submitted to ICLR 2022 (\url{https://github.com/fedebotu/ICLR2022-OpenReviewData}). Each paper typically received three to five reviews, enabling the computation of its average score and standard deviation. Each point in this figure corresponds to a possible (average score, standard deviation) pair, where the size of a point reflects the number of papers with this pair of average review score and standard deviation. \label{fig:iclr}}
\end{figure}

To assess submission quality based on review scores from a statistical perspective, we begin by assuming that the review scores follow an exponential family distribution in the canonical form:
\begin{equation}\label{eq:exp}
p_{\theta}(x) = \mathrm{e}^{\theta x-b\left(\theta \right)} c\left(x\right),
\end{equation}
where $\theta$ is a natural parameter, $b$ is a log-partition function, and $c$ is a carrier function. Here, $\theta$ or, equivalently, $\mu := \EE [X] = b'(\theta)$ can be considered as the ground-truth quality of the paper.\footnote{This ground-truth quality can be understood in practical terms as the expected review score if the paper were to be evaluated by a large number of reviewers.} Exponential family distributions provide a simple yet flexible approach for modeling the dependence between the mean and the variance of a random variable, as is the case in Figure~\ref{fig:iclr}. The most relevant example is the Binomial distribution, whose data type (taking integer values within a certain range) and variance-mean dependence (smaller variance for high and low quality papers) makes it an ideal choice for modeling noisy review scores. For an author who submits $n \ge 2$ papers to a machine learning conference, the average review score $X_i$ of the $i$th paper is thought of as a realization from the exponential family distribution with parameter $\theta^\star_i$:
\[
X_i \sim  p_{\theta_i^\star}(x)
\]
for $i = 1, \ldots, n$.

A readily available estimator of the quality $\mu_i^\star = \EE X_i = b'(\theta^\star_i) $ of the $i$th paper is simply the raw review score $X_i$, which is the maximum likelihood estimator (MLE) of $\mu_i^\star$. To see this, note that the negative log-likelihood reads:
\[
\log \left( p_{\theta_1}(X_1) \cdots p_{\theta_n}(X_n) \right) = \sum_{i=1}^n \left[ \theta_{i}X_{i}-b\left(\theta_{i}\right)+\log c\left(X_{i}\right) \right],
\] 
where the review scores $X_1, \ldots, X_n$ are assumed to be independent in light of the fact that the review processes of the $n$ papers are independent.\footnote{Due to the large numbers of submissions and reviewers, it is unlikely that different papers by the same author would be reviewed by the same reviewer.} Minimizing the above negative log-likelihood over $\theta_i$ yields $\widehat{\theta}_{i}=(b')^{-1}(X_{i})$, which is equivalent to $\widehat{\mu}_{i}=X_{i}$. The triviality of the MLE in this case arises from treating the $n$ submissions separately, despite the fact that these $n$ papers share a \textit{common} author.

To obtain a better estimation of submission quality, ICML 2023, for the first time in its history, asked any authors who submitted multiple papers to \textit{rank} their submissions from their own perspective, upon submission confirmation (\url{https://icml.cc/Conferences/2023/CallForPapers}, \url{https://openrank.cc}). This experiment is based on a recent proposal by \cite{su2021you, su2022truthful} to improve peer review in machine learning conferences. Formally, let an author of $n$ submissions report a ranking of $\pi$ such that, as claimed by the author, $\pi$ sorts the ground-truth natural parameters $\theta_1^\star, \ldots, \theta_n^\star$ in descending order.\footnote{Mathematically speaking, $\pi(1), \ldots, \pi(n)$ is a permutation of $1, \ldots, n$.} It is important to note, however, that the author might report an incorrect ranking, and therefore, $\pi$ is not necessarily identical to the ground-truth ranking $\pi^\star$ that satisfies
\begin{equation}\label{eq:constraint}
\theta_{\pi^\star(1)}^{\star} \geq \theta_{\pi^\star(2)}^{\star} \geq \cdots \geq \theta_{\pi^\star(n)}^{\star}.
\end{equation}

With the author-provided ranking in place, in this paper, we introduce the ranking-constraint MLE to estimate $\mu^\star_1, \ldots, \mu^\star_n$: 
\begin{equation}\label{eq:generalized-isotonic-regression}
\begin{aligned}
\mathop{\arg\min}_{\theta_1, \ldots, \theta_n}\qquad & \sum_{i=1}^{n}\left[-\theta_{i}X_{i}+b\left(\theta_{i}\right)\right]\\
\text{s.t.}\qquad & \theta_{\pi(1)}\geq\theta_{\pi(2)}\geq\cdots\geq\theta_{\pi(n)}.
\end{aligned}
\end{equation}
This is a convex program because $b(\cdot)$ is convex, and it possesses a unique minimizer $\widehat{\boldsymbol{\theta}}=(\widehat{\theta}_{1},\ldots,\widehat{\theta}_{n})$ in the general setting where $b(\cdot)$ is strictly convex. The estimated scores are
\begin{equation}\label{eq:modify}
\widehat{\boldsymbol{\mu}} = b'(\widehat{\boldsymbol{\theta}}) = \left( b'(\widehat{\theta}_{1}), \ldots, b'(\widehat{\theta}_{n}) \right).
\end{equation}

The implementation of the mechanism \eqref{eq:generalized-isotonic-regression} requires knowing the log-partition function $b(\cdot)$, which is also the case for obtaining the adjusted scores in \eqref{eq:modify}. Interestingly, we show that this dependence can be circumvented, thereby making the mechanism \textit{robust} to model misspecification. Specifically, by using a classical result from \citet[Theorem 3.1]{barlow1972isotonic}, \eqref{eq:generalized-isotonic-regression} is equivalent to isotonic regression 
\begin{equation}\label{eq:isotonic-regression-no-b}
\begin{aligned}
\mathop{\arg\min}_{\mu_1, \ldots, \mu_n} \qquad & \sum_{i=1}^{n}\left(X_{i}-\mu_{i}\right)^{2}\\
\text{s.t.}\qquad & \mu_{\pi(1)}\geq\mu_{\pi(2)}\geq\cdots\geq\mu_{\pi(n)},
\end{aligned}
\end{equation}
which is referred to as the \textit{Isotonic Mechanism}~\citep{su2021you}, in the sense that the adjusted score vector $\widehat\bmu = b'(\widehat\btheta)$ in \eqref{eq:modify} coincides with the solution to \eqref{eq:isotonic-regression-no-b}. The feasible set of \eqref{eq:isotonic-regression-no-b} contains the ground-truth parameter $\bmu^\star = (\mu^\star_1, \ldots, \mu^\star_n)$ when $\pi = \pi^\star$, that is, when the author is \textit{truthful}. In particular, the feasible sets of \eqref{eq:generalized-isotonic-regression} and \eqref{eq:isotonic-regression-no-b} are essentially the same in view of the relationship $\mu = b'(\theta)$ and the mapping $b'(\cdot)$ is nondecreasing. This equivalence is convenient because it enables the implementation \textit{without} requiring the knowledge of the probability distribution of the observations $X_i$'s. 

The Isotonic Mechanism can be better understood within the framework of mechanism design~\citep{myerson1989mechanism}. In this context, the author acts as the sole agent possessing private information about the quality of all her papers submitted to a conference. The author provides a ranking that presumably sorts her papers from best to worst to the conference organizers---referred to as the principal. The principal then inputs this ranking, along with the review scores, into the Isotonic Mechanism to obtain adjusted review scores. It is important to note that reviewers assess the papers as they would in any standard machine learning conference, without knowledge of the author-provided rankings. Similarly, the author has no access to the review scores before submitting her ranking.

The Isotonic Mechanism \eqref{eq:isotonic-regression-no-b} is \textit{owner-assisted} \citep{su2022truthful}, as it requires authors to provide information pertaining to their own items. For the practical application of this mechanism, several questions immediately arise:

\begin{enumerate}
\item
First, might the author provide an inaccurate ranking despite being aware of the ground-truth ranking?

\item
Is it possible to elicit any other truthful information from the author? For example, can conference organizers inquire how \textit{much} one paper surpasses another of the same author in terms of quality?

\item
Lastly, would the estimator derived from the Isotonic Mechanism be more accurate than the raw scores $X_i$'s? If so, is it optimal in some sense?
\end{enumerate}

A challenge in addressing these questions stems from the \textit{generality} of exponential family distributions, including normal, binomial, and Poisson distributions, among others. In the subsequent part of this section, we present the main results of our paper that respond to these three questions. The proofs of the results in Sections~\ref{subsec:truthfulness}, \ref{sec:necess-pairw-comp}, and \ref{sec:optimal-estimation} are given in Sections~\ref{sec:truthfulness}, \ref{sec:rank-almost-optim}, and \ref{sec:proof-minimax}, respectively, accompanied by additional findings.


\subsection{Truthfulness}
\label{subsec:truthfulness}


To analyze the strategy of a \textit{rational} author, it is imperative to specify the author's utility. We assume an additive utility function across all coordinates of the adjusted review score vector $\widehat\bmu = (\widehat\mu_1, \ldots, \widehat\mu_n)$. This assumption stems from the observation that, given the anonymity of the process and the large number of submissions, the decision-making processes\footnote{ICML 2023 collects rankings but will not utilize adjusted scores in decision-making. Adopting such a significant reform in peer review necessitates time for such a large community.} for each paper are largely independent of one another. This rationale informs the subsequent assumption, as in~\cite{su2022truthful}:
\begin{assumption}[Utility]\label{ass:convex}
The author's overall utility takes the form
\begin{equation}\label{eq:util_form}
\mathsf{Util} := \EE \left[ \sum_{i=1}^n U( \widehat \mu_i ) \right]
\end{equation}
for some nondecreasing convex function $U$. 
\end{assumption}

The author is rational, which means that she will  report the ranking that maximizes her expected utility, whether truthfully or otherwise. The convexity assumption acknowledges that a higher score may increase the likelihood of winning oral presentations or even best paper awards, thus yielding progressively greater payoffs. This is particularly applicable in the long term, as reviews are often made public at machine learning conferences such as ICLR. Consequently, a higher review score significantly enhances the submission's impression within the community. Additional justification can be found in Section 4 of \cite{su2021you} and Section 2 of \cite{su2022truthful}.

\begin{assumption}\label{ass:author2}
The author knows the ground-truth ranking $\pi^\star$ that fulfills \eqref{eq:constraint} or, equivalently, $\mu_{\pi^\star(1)}^{\star} \geq \mu_{\pi^\star(2)}^{\star} \geq \cdots \geq \mu_{\pi^\star(n)}^{\star}$.
\end{assumption}

For this assumption, the author merely needs to understand the relative quality of her submissions rather than their precise quality.

The subsequent theorem examines a \textit{regular} exponential family distribution $p_{\theta}(x) = \mathrm{e}^{\theta x-b\left(\theta \right)} c\left(x\right)$, where $b(\cdot)$ is twice differentiable and satisfies $b''(\theta) > 0$ within its domain. This condition implies that $X$ has nonzero and finite variance due to the identity $\mathsf{Var}\left(X\right)=b''\left(\theta\right)$. Lastly, we presume that $X_{1},\ldots,X_{n}$ are independently generated by the exponential family distribution with distinct parameters. 

Let $\mathsf{Util}(\pi)$ denote the expected overall utility \eqref{eq:util_form} of the solution to the Isotonic Mechanism \eqref{eq:isotonic-regression-no-b} when the input ranking is $\pi$.

\begin{theorem}\label{thm:truth-telling}
Under Assumptions \ref{ass:convex} and \ref{ass:author2}, the author maximizes her expected overall utility by truthfully reporting the ground-truth ranking $\pi^\star$. That is,
\[
\mathsf{Util}(\pi^\star) \ge \mathsf{Util}(\pi)
\]
for any ranking $\pi$.
\end{theorem}

An important feature of this theorem is that it applies to the broad class of exponential family of distributions. In game theoretic terms, this theorem demonstrates that the owner-assisted mechanism \eqref{eq:isotonic-regression-no-b} is incentive-compatible. If the utility function $U$ is strictly convex, it satisfies strict incentive compatibility, which means that $\mathsf{Util}(\pi^\star) > \mathsf{Util}(\pi)$ for any ranking $\pi$ that does not sort $\btheta^\star = (\theta^\star_1, \ldots, \theta^\star_n)$ or $\bmu^\star$ in descending order.

\begin{figure}[t]
	\centering
	
	\begin{tabular}{cc}
		\includegraphics[scale=0.45]{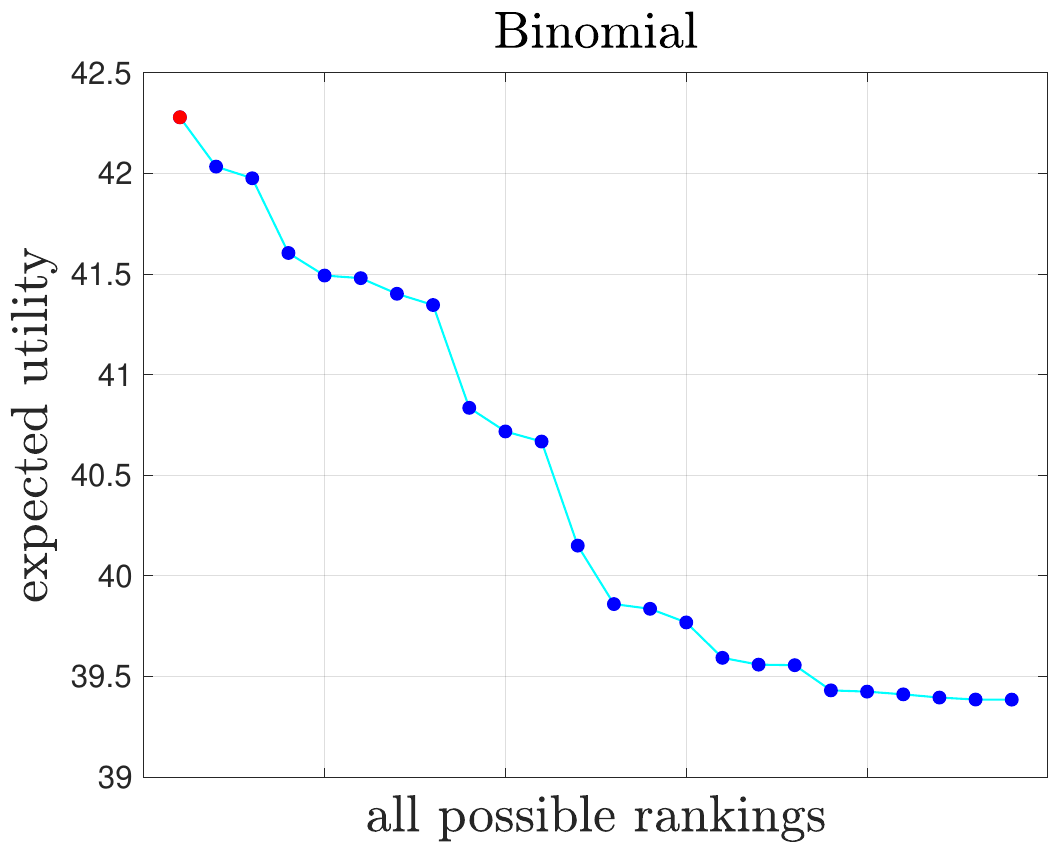} & \includegraphics[scale=0.45]{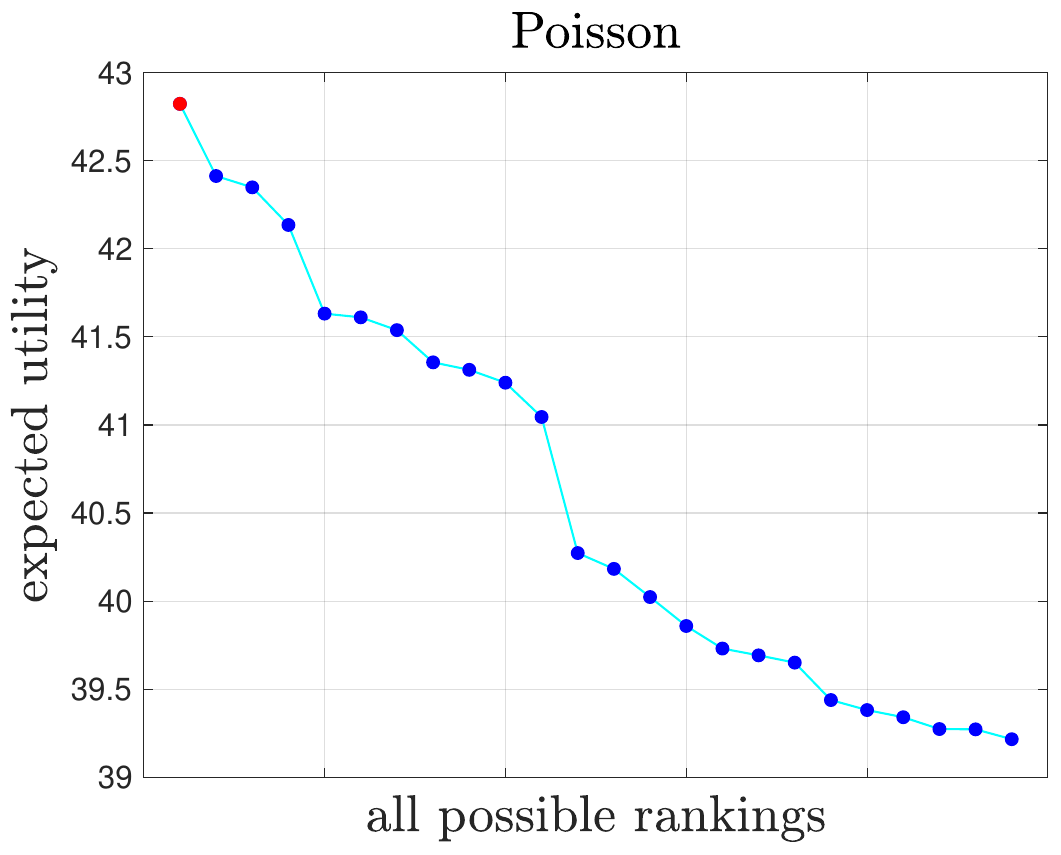}\tabularnewline
		$\quad\quad$(a) & $\quad\quad$(b)\tabularnewline
	\end{tabular}
	\caption{Expected overall utility for all $n! = 4! = 24$ possible rankings in descending order, averaged over $100$ runs. The left panel uses review scores simulated from binomial distributions, while the right panel uses review scores simulated from Poisson distributions. In both panels, $(\mu_{1}^{\star},\mu_{2}^{\star},\mu_{3}^{\star},\mu_{4}^{\star})=(8,7,6,4)$, and the utility function $U(x) = \max\{x, 0\}^2$ (binomial counts are generated from $\mathrm{Binom}(10, \mu_{i}^{\star}/10)$). The red dot represents the expected utility when the author reports truthfully. The second to fourth highest expected overall utilities are achieved by the rankings $(2,1,3,4)$, $(1,3,2,4)$, and $(1,2,4,3)$ for binomial distributions, and $(1,3,2,4)$, $(2,1,3,4)$, and $(1,2,4,3)$ for Poisson distributions, respectively. These results suggest that the mechanism may be stable against slightly misspecified rankings.}
\label{fig:utility}
\end{figure}

A numerical illustration of this theorem is presented in Figure~\ref{fig:utility}. Consider a scenario with four submissions whose true scores are $(\mu_{1}^{\star},\mu_{2}^{\star},\mu_{3}^{\star},\mu_{4}^{\star})=(8,7,6,4)$. For each submission, three independent review scores are generated from either a binomial (with 10 trials) or Poisson distribution, and their average is used as the original review score. This figure exhibits the author's expected utility for all $4!=24$ possible rankings. As observed for both distributions, the expected utility is maximized when the author provides the true ranking $\pi^{\star}=(1,2,3,4)$.

While Theorem~\ref{thm:truth-telling} offers desired guarantees of truthfulness, the author might not be able to know the complete ranking of her papers. This motivates us to consider the scenario when, instead, the author knows only a \textit{coarse ranking}~\citep{su2022truthful}: given positive integers $n_{1},\ldots,n_{p}$ satisfying $n_{1}+\cdots+n_{p}=n$, let blocks $I_{1},\ldots,I_{p}$ be a partition of the indices ${1, \ldots, n}$ such that the size $|I_j| = n_{j}$ for $1 \le j \le p$; a vector $\ba \in \R^n$ is said to have coarse ranking $\bm{I}=(I_{1},\ldots,I_{p})$, denoted as
\[
\ba_{I_1} \ge \cdots \ge \ba_{I_p},
\]
if all entries of $\ba_{I_i}$ are larger than or equal to all entries of $\ba_{I_j}$ as long as $1 \le i < j \le p$.

In a variant of the Isotonic Mechanism, we require any author to provide a coarse ranking of her submissions, which amounts to replacing the second line of \eqref{eq:generalized-isotonic-regression} by
\[
\btheta_{I_1} \geq \btheta_{I_2} \geq \cdots \geq \btheta_{I_p}.
\]
This is equivalent to the following coarse Isotonic Mechanism:
\begin{equation}\label{eq:isotonic-regression-coarse}
\begin{aligned}
\mathop{\arg\min}_{\mu_1, \ldots, \mu_n}\qquad & \sum_{i=1}^{n}\left(X_{i}-\mu_{i}\right)^{2}\\
\text{s.t.}\qquad & \bmu_{I_1}\geq \bmu_{I_2} \geq \cdots\geq\bmu_{I_p}.
\end{aligned}
\end{equation}
When $n_1 = n_2 = \cdots = n_n = 1$, this mechanism reduces to the original Isotonic Mechanism. In general, the coarse Isotonic Mechanism provides the author with certain flexibility in that the author does not need to provide pairwise comparisons between some papers. For example, if $n_1 = 1$ and $n_2 = n - 1$, then the author is asked to report their best paper. In another example, $n_1 = \lfloor n/3 \rfloor$ and $n_2 = \lceil 2n/3 \rceil$, where $\lfloor \cdot \rfloor$ and $\lceil \cdot \rceil$ denote the floor and ceiling functions, respectively. The author is asked to report only the top one-third of her submissions.

The coarse Isotonic Mechanism is truthful, as we show in the following theorem. In this theorem, denote by $\mathsf{Util}(\bI)$ the expected overall utility of the solution to the coarse Isotonic Mechanism \eqref{eq:isotonic-regression-coarse}; define $\bm{I}^{\star}=(I_{1}^{\star},\cdots,I_{p}^{\star})$ as
\[
I_{j}^{\star}=\biggl\{\pi^{\star}\left(i\right): 1 + \sum_{k=1}^{j-1}n_{k} \leq i\leq\sum_{k=1}^{j}n_{k}\biggr\}
\]
for $j =1,2,\ldots,p$. This is the ground-truth coarse ranking with respect to $\bmu^\star$, which satisfies
\[
\bmu_{I_1^\star}^\star\geq \bmu_{I_2^\star}^\star \geq \cdots\geq\bmu^\star_{I_p^\star}
\]

\begin{theorem} \label{thm:coarse-optimality}
Under Assumptions \ref{ass:convex} and \ref{ass:author2}, the author maximizes her expected overall utility by truthfully reporting the coarse ground-truth ranking $\bI^\star$. That is,
\[
\mathsf{Util}(\bI^\star) \ge \mathsf{Util}(\bI)
\]
for any coarse ranking $\bI$ of size $(n_1, \ldots, n_p)$.
\end{theorem}

It is worth mentioning that in the peer review process, the author decides both the coarseness $(n_1,\ldots,n_p)$ and the reported coarse ranking. The above theorem asserts that for a given coarseness, the author maximizes her utility by truthfully reporting the coarse ranking, but it says nothing about the decision of the coarseness (e.g., the author might claim that the quality of her $n$ papers are indistinguishable).  In order to elicit as much truthful information as possible, in practice, the conference organizer might set a lower bound for $p$ for an author that depends on her number of submissions $n$.

\subsection{Is ranking optimal?}
\label{sec:necess-pairw-comp}

While we have shown truthfulness of eliciting a ranking or coarse ranking from the author, we aim to investigate the feasibility of truthfully eliciting other types of information for estimation. To achieve this, it is necessary to first formalize the types of information that conference organizers can elicit. Specifically, the $n$-dimensional Euclidean space $\R^n$, where the ground-truth vector $\bmu^\star$ resides, is partitioned into $\mathcal S := \{S_1, S_2, \ldots \}$ such that
\[
\bigcup_{S \in \mathcal S} S = \R^n, \quad \text{ and } S \bigcap S' = \emptyset
\]
for any $S, S' \in \mathcal S$ but $S \ne S'$. The cardinality of $\mathcal S$ can be either finite or infinite. As a collection of sets, $\mathcal S$ is referred to as a knowledge partition and a set $S \in \mathcal S$ is called a knowledge element~\citep{su2022truthful}. By definition, there must exist a unique knowledge element $S^\star$ that contains the ground-truth vector~$\bmu^\star$.

We focus on a class of knowledge partitions that are cut by hyperplanes. Specifically, a knowledge partition is cut by $(n-1)$-dimensional hyperplanes $H_1, \ldots, H_K$ if any knowledge element is enclosed by some of $H_1, \ldots, H_K$ and no other hyperplane passes through its interior. Such a knowledge partition is called \textit{linear}.


In the following variant of Assumption~\ref{ass:author2}, we assume that the author knows which is the \textit{true} knowledge element $S^\star$. This assumption reduces to Assumption~\ref{ass:author2} when $\mathcal S$ is the collection of all rankings.

\setcounter{assumptions}{1}
\renewcommand{\theassumptions}{\arabic{assumptions}'}
\begin{assumptions}\label{ass:author_e}
The author possesses sufficient understanding of her submissions to identify the specific knowledge element within the linear knowledge partition $\mathcal S$ that contains the ground-truth $\bmu^\star$.

\end{assumptions}

In this mechanism, the author is mandated to choose an element, say, $S$, from the linear knowledge partition $\mathcal S$ and claim that $S$ contains the ground-truth vector. This step is done before she observes the review scores $\bX = (X_1, \ldots, X_n)$. As earlier, $S$ may or may not be the true knowledge element $S^\star$. The conference organizers solve the following shape-restricted regression:\footnote{This optimization may be convex or not, depending on whether the set $S$ is convex or not.}
\begin{equation}\label{eq:projection-knowledge-partition}
\begin{aligned}
\mathop{\arg\min}_{\mu_1, \ldots, \mu_n}\qquad & \sum_{i=1}^{n}\left(X_{i}-\mu_{i}\right)^{2}\\
\text{s.t.}\qquad & \bmu\in S.
\end{aligned}
\end{equation}
Denote by $\widehat{\bmu}^S$ the unique minimizer of this convex optimization problem. As in Assumption~\ref{ass:convex}, the author seeks to maximize her expected overall utility
\begin{equation}\nonumber
\mathsf{Util}(S)=\mathbb{E}\left[\sum_{i=1}^{n}U(\widehat{\mu}^S_{i})\right]
\end{equation}
by reporting any knowledge element $S$ in the linear knowledge partition $\mathcal{S}$. We are interested in the class of linear knowledge partitions that incentivize the author to tell the truth.

\begin{definition} \label{defn:truthful}
A linear knowledge partition $\mathcal{S}$ is said to be truthful if, under Assumptions \ref{ass:convex} and \ref{ass:author_e}, the author could always maximize her expected overall utility by reporting the ground-truth knowledge element. That is,
\[
\mathsf{Util}(S^{\star})\geq\mathsf{Util}(S)
\]
for any $S \in \mathcal S$.

 \end{definition}

\begin{remark}
This truthfulness condition requires that it holds for any possible values of the parameters $\theta_1, \ldots, \theta_n$ in the exponential family distribution~\eqref{eq:exp}.
\end{remark}

In a nutshell, Theorems~\ref{thm:truth-telling} and \ref{thm:coarse-optimality} say that the collections of all rankings and, more generally, of all coarse rankings, are truthful. As a defining feature of these knowledge partitions, any (coarse) ranking is enclosed by \textit{pairwise-comparison} hyperplanes $x_i - x_j = 0$ for some $1 \le i < j \le n$. For better estimation, however, more fine-grained knowledge partitions are preferred as they narrow down the ``search space'' $S$ in \eqref{eq:projection-knowledge-partition}. Specifically, does there exist a truthful linear knowledge partition that is more fine-grained than rankings?

Our next main result shows that under certain regularity conditions, which are satisfied by many exponential family distributions as we will show in Section~\ref{sec:rank-almost-optim}, any linear truthful knowledge partition must be cut by pairwise-comparison hyperplanes up to an intercept.

\begin{theorem}\label{thm:necessary-condition-slope}
Assume that the image of $b'(\cdot)$ is $\mathbb{R}$, and
\begin{equation} \label{eq:growth-condition}
	\lim_{\theta\to-\infty} \frac{b'(\theta)}{\sqrt{b''(\theta)}} \to -\infty,\qquad \lim_{\theta\to\infty} \frac{b'(\theta)}{\sqrt{b''(\theta)}} \to \infty.
\end{equation} 
Consider any linear knowledge partition $\mathcal{S}$ cut by hyperplanes $H_1,\ldots,H_K$. If $\mathcal{S}$ is truthful, then for each $1\leq k \leq K$, $H_k$ must takes the form $x_{i_k} - x_{j_k} = \mathsf{const}$ for some $1 \le i_k < j_k \le n$.
\end{theorem}

This theorem does not specify the intercepts of these hyperplanes. Although we hypothesize that these intercepts should be $0$, presenting a proof for general exponential family distributions seems difficult, and the analysis might require examination on a case-by-case basis. For instance, the following proposition illustrates that in the case of normal distributions, the intercepts must be $0$.

\begin{prop}\label{thm:necessary-condition-intercept}
Suppose $X_{i}\sim\mathcal{N}(\mu_{i}^{\star},\sigma^{2})$ independently for $i = 1, \ldots, n$, with $\sigma^{2}>0$. 
Consider any linear knowledge partition $\mathcal{S}$ cut by hyperplanes $H_1,\ldots,H_K$. If $\mathcal{S}$ is truthful, then for each $1\leq k \leq K$, $H_k$ must be a pairwise-comparison hyperplane that takes the form $x_{i_k} - x_{j_k} = 0$ for some $1 \le i_k < j_k \le n$.
\end{prop}

These results demonstrate that pairwise comparisons are essential for truthfulness, and the author must consider the principle that ``you can't have your cake and eat it too'' when providing information about her submission quality. The author can make pairwise comparisons by knowing any arbitrary monotone transformation of the coordinates of the ground-truth $\bmu^\star$, without needing to know the individual values.

All $n(n-1)/2$ pairwise-comparison hyperplanes produce $n!$ isotonic cones, each of which is parameterized by a ranking $\pi$ and takes the form ${\bmu: \mu_{\pi(1)} \ge \cdots \ge \mu_{\pi(n)}}$. No isotonic cone can be further divided by pairwise-comparison hyperplanes. Consequently, this theorem implies that the collection of rankings is the most fine-grained truthful linear knowledge partition, at least in the case of normally distributed scores. Whenever the author knows the ground-truth ranking or, equivalently, Assumption~\ref{ass:author2} is satisfied, the Isotonic Mechanism is optimal among all estimators of the form \eqref{eq:projection-knowledge-partition}.

\subsection{Optimal estimation}
\label{sec:optimal-estimation}

It remains unclear whether the Isotonic Mechanism, along with its truthfulness, enhances estimation. If so, quantifying the extent it surpasses the original review scores $\bX$ is vital. We begin addressing these questions by presenting the following simple result.

\begin{prop} \label{prop:projection}
Under Assumptions \ref{ass:convex} and \ref{ass:author2}, the solution $\widehat\bmu$ to the Isotonic Mechanism \eqref{eq:isotonic-regression-no-b} satisfies
\[
\EE \|\widehat\bmu - \bmu^\star\|^2 \le \EE \|\bX - \bmu^\star\|^2,
\]
where, throughout the paper, $\|\cdot\|$ denotes the $\ell_2$ norm.
\end{prop}

\begin{remark}
Due to the truthfulness of the mechanism, $\widehat\bmu$ is the solution to \eqref{eq:isotonic-regression-no-b} given the ground-truth ranking $\pi = \pi^\star$. The other results in this subsection are also under Assumptions \ref{ass:convex} and \ref{ass:author2}, but this fact will not be reiterated for the sake of simplicity. This theorem is also applicable to the coarse Isotonic Mechanism.
\end{remark}

Moreover, the improvement can be considerable in a certain regime. We assume that all coordinates $\mu_1^\star, \ldots, \mu_n^\star$ are lower and upper bounded by $V_{\min}$ and $V_{\max}$, respectively. This supplementary assumption is not restrictive when implementing the Isotonic Mechanism in machine learning conferences. For instance, all review scores must range from 1 to 10 for NeurIPS 2022 and 1 to 8 for ICML 2023.
\begin{prop} \label{prop:consistency}
Let the ground-truth scores $\mu_1^\star, \ldots, \mu_n^\star$ be arbitrary and satisfy $V_{\min} \le \mu_1^\star, \ldots, \mu_n^\star \le V_{\max}$ for constants $V_{\min} < V_{\max}$. As $n \to \infty$, the Isotonic Mechanism satisfies
\[
\limsup_{n \to \infty} \frac1{n}\EE \|\widehat\bmu - \bmu^\star\|^2 = 0.
\]

\end{prop}

This proposition demonstrates that the (truthful) ranking constraint of the Isotonic Mechanism allows for estimation consistency. In stark contrast, the original review scores generally obey
\[
\frac1n \EE \|\bX - \bmu^\star\|^2 =  \frac{\sum_{i=1}^n b''(\theta_i^\star)}{n} \centernot\longrightarrow 0
\]
in the limit $n \to \infty$.

Next, we establish that the Isotonic Mechanism not only surpasses the original scores but also attains nearly minimax estimation among all estimators under certain conditions. We employ $\sup_{V_{\min} \le \bmu^{\star} \le V_{\max}}$ as a shorthand for $\sup_{V_{\min} \le \mu^{\star}_i \le V{\max}, \; \forall 1 \le i \le n}$. Let
\[
\sigma^{2}\coloneqq\max_{\theta_{\min}\leq\theta\leq\theta_{\max}}b''\left(\theta\right),
\]
where $\theta_{\max}=(b')^{-1}(V_{\max})$ and $\theta_{\min}=(b')^{-1}(V_{\min})$. The quantity $\sigma^2$ represents the highest variance of review scores when the ground-truth score resides within $[V_{\min},V_{\max}]$.

\begin{theorem}\label{thm:general_minimax_informal} 
Under certain regularity conditions, there exist two constants $C_{1},C_{2}>0$ such that 
\begin{equation}
		\sup_{V_{\min} \le \bmu^{\star} \le V_{\max}}\,\mathbb{E} \left\Vert \widehat{\bmu}-\bm{\mu}^{\star}\right\Vert^{2}\leq C_{1}n\sigma^{2}\min\left\{ 1,\left(\frac{V_{\max}-V_{\min}}{n\sigma}\right)^{2/3}+\frac{\log n}{n}\right\} ,\label{eq:general_minimax_upper_informal}
\end{equation}
where $\widehat{\bm{\mu}}$ is the optimal solution to (\ref{eq:isotonic-regression-no-b}), and
\begin{equation}
		\inf_{\widetilde{\bm{\mu}}} \sup_{V_{\min} \le \bmu^{\star} \le V_{\max}} \mathbb{E}\left\Vert \widetilde{\bm{\mu}}-\bm{\mu}^{\star}\right\Vert^{2}\geq C_{2}n\sigma^{2}\min\left\{ 1,\left(\frac{V_{\max}-V_{\min}}{n\sigma}\right)^{2/3}+\frac{1}{n}\right\} ,\label{eq:general_minimax_lower_informal}
\end{equation}
where the infimum is taken over all estimators of the true mean vector $\bm{\mu}^{\star}$ using data $X_{1},\ldots,X_{n}$. 
	
\end{theorem}

The formal statement of Theorem~\ref{thm:general_minimax_informal} is provided in Theorem \ref{thm:general_minimax} in Section \ref{sec:proof-minimax}. The regularity condition is mild and can be satisfied by many exponential family distributions, including Gaussian, binomial, Poisson, and Gamma distributions. The upper bound \eqref{eq:general_minimax_upper_informal} is developed in \citet{zhang2002risk}. We obtain the lower bound \eqref{eq:general_minimax_lower_informal} by extending the proof technique for the additive Gaussian noise setting in \citet{bellec2015sharp} to general exponential family distributions.

Theorem \ref{thm:general_minimax_informal} demonstrates that the Isotonic Mechanism achieves minimax-optimal estimation, at most up to a logarithmic factor. This logarithmic factor would not be in effect when the total variation $V_{\max}-V_{\min}$ remains constant and $n$ is large. The boundedness of the total variation indeed holds for peer review in machine learning conferences, as the scores must be integers between 1 and, for example, 8 or 10. In this case, the minimax estimation risk is $O((V_{\max}-V_{\min})^{\frac23} n^{\frac13} \sigma^{\frac43})$, and this mechanism precisely achieves this rate. For comparison, the $\ell_2$ risk of using the original scores $\bX$ is $n \sigma^2$. This shows that the Isotonic Mechanism significantly outperforms the original scores in terms of estimation when the number of submissions $n$ is large and the noise level $\sigma$ is high.

\begin{figure}[t]
	\centering
	
	\begin{tabular}{cc}
		\includegraphics[scale=0.45]{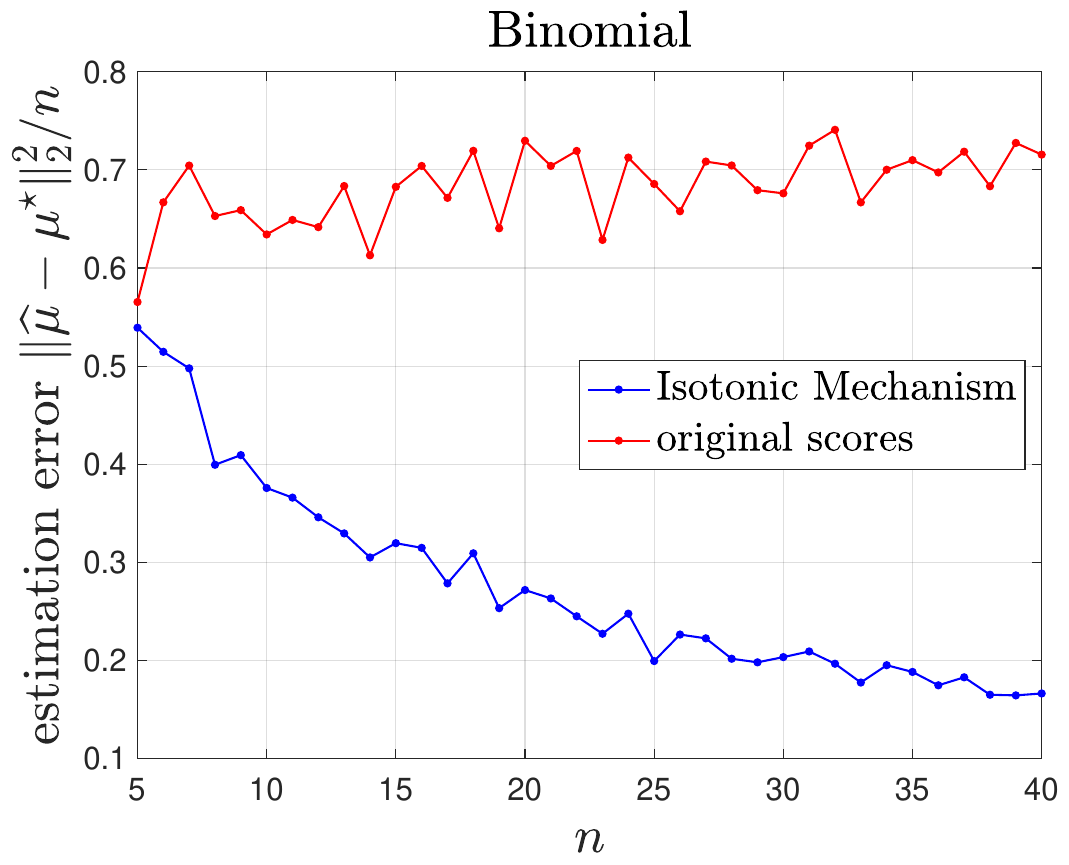} & \includegraphics[scale=0.45]{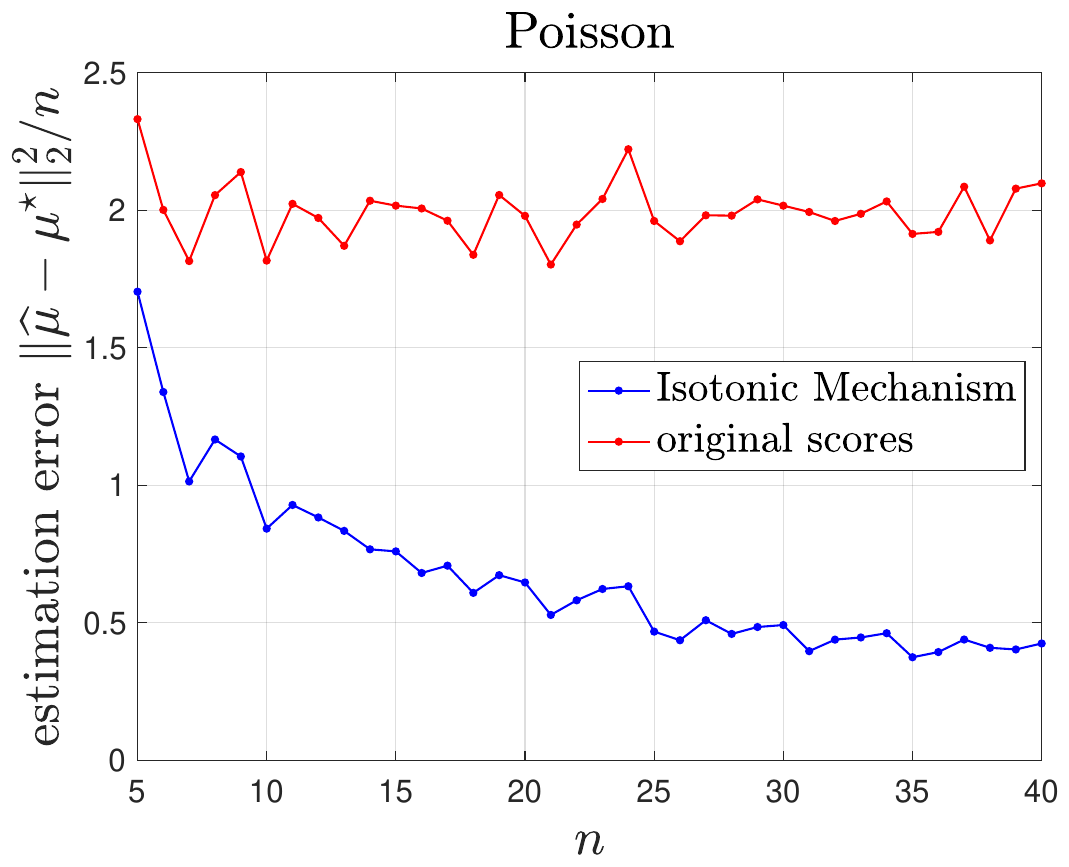}\tabularnewline
		$\quad\quad$(a) & $\quad\quad$(b)\tabularnewline
	\end{tabular}
	
	\caption{Estimation errors $\Vert\widehat{\bm{\mu}}-\bm{\mu}^{\star}\Vert^{2}/n$ for the Isotonic Mechanism and $\Vert\bX-\bm{\mu}^{\star}\Vert^{2}/n$ for the original review scores, with varying numbers of submissions. The left panel presents results when the review scores follow binomial distributions, while the right panel exhibits results when the review scores follow Poisson distributions, in the same setting as the numerical experiment in Section \ref{subsec:truthfulness}. For any $n$, we set $\mu_{i}^{\star}=9-6(i-1)/(n-1)$ for $i =1, \ldots, n$.}
\label{fig:estimation_err}
\end{figure}

Figure \ref{fig:estimation_err} presents a numerical experiment comparing the estimation accuracy of the Isotonic Mechanism and the averaged review scores, using the same setting as in Section \ref{subsec:truthfulness}. The Isotonic Mechanism is significantly more accurate than the review scores.

\subsection{Related work}
\label{sec:related-work}

The declining quality of peer review in machine learning conferences has been increasingly recognized, resulting in a flurry of research activity toward improving reviewer assignments \citep{kobren2019paper,jecmen2020mitigating,leyton2024matching}, and bias reduction \citep{paul1981bayesian,mackay2017calibration,wang2018your}, among other things; see \citet{shah2022challenges} for a survey.

The most related work is \cite{su2021you} and \cite{su2022truthful}, which consider the same problem as the present paper, but in the setting where the review scores follow
\begin{equation}\label{eq:add_noise}
X_i = \mu_i^\star + z_i
\end{equation}
with i.i.d.~noise terms $z_1, \ldots, z_n$. Specifically, \cite{su2021you} proposes the Isotonic Mechanism and proves its truthfulness when applied to \eqref{eq:add_noise}. In \cite{su2022truthful}, the author further shows that pairwise comparisons are necessary for the truthfulness of the general mechanism \eqref{eq:projection-knowledge-partition} in the setting of the additive-noise model \eqref{eq:add_noise}.

However, it is worth noting that the techniques developed in \cite{su2021you} and \cite{su2022truthful} are not sufficient to cope with general exponential family distributions. In more detail, the proof of truthfulness in \cite{su2021you} for the additive-noise model \eqref{eq:add_noise} heavily relies on a coupling technique based on the exchangeability of the noise terms. In the case of general exponential family distributions, the noise distribution and particularly, the variance of the noise, vary with the mean. Our proof of Theorems~\ref{thm:truth-telling} and \ref{thm:coarse-optimality} employs a novel technique that addresses the mean-variance dependence.

Moreover, the proof of Theorem~\ref{thm:necessary-condition-slope} is much more sophisticated for exponential family distributions than that for the additive-noise model~\eqref{eq:add_noise} in \cite{su2022truthful}. This distinction arises from the mean-variance dependence in exponential family distributions that prevents us from constructing certain solutions to the general mechanism \eqref{eq:projection-knowledge-partition}, essential to the proof in \cite{su2022truthful}. To tackle this challenge, we leverage a technique based on symmetric distributions, which may be a contribution of independent interest.

From a different angle, our approach connects to the literature on cheap talk \citep{chakraborty2007comparative} in game theory and mechanism design. Cheap-talk games consider a privately informed sender (author) whose message (or ``talk'') does not directly affect payoffs, rendering the sender's statements non-binding. However, in our setup, the receiver (conference organizers) commits to utilizing the reported ranking in the Isotonic Mechanism to adjust scores that \emph{do} affect payoffs. Thus, the author's disclosure cannot be considered ``cheap'' as the ranking is binding through its direct impact on final outcomes. Moreover, the conference organizers also observe review scores, which further differentiates our work from the standard setting of cheap talk. While our mechanism shares the spirit of ``information elicitation'' inherent in cheap talk, both the binding nature of the final score adjustment and the presence of additional feedback (the $X_i$'s) position the Isotonic Mechanism beyond the scope of cheap talk.



%% file: truth.tex
\section{Truthfulness}
\label{sec:truthfulness}

In this section, we prove Theorems~\ref{thm:truth-telling} and \ref{thm:coarse-optimality}. As we shall see soon, our proofs make use of several technique novelties that are not found in \cite{su2021you, su2022truthful} in order to cope with the dependence between mean and variance for exponential family distributions.

\subsection{Preliminaries}

We start by collecting some elements that will be useful for the proofs.

For any vector $\bm{a}\in\mathbb{R}^{n}$, let 
\[
a_{[1]}\geq a_{[2]}\geq\cdots\geq a_{[n]}
\]
denote the components of $\bm{a}$ in descending order. Let $\mathcal{C}$
denote the standard isotonic cone 
\begin{equation}
\mathcal{C} \coloneqq\left\{ \bm{x}\in\mathbb{R}^{n}:x_{1}\geq x_{2}\geq\cdots\geq x_{n}\right\}.\label{eq:C-defn}
\end{equation}
As is clear, $\mathcal{C}$ is a convex cone. Let $\mathcal{P}_{\mathcal{C}}:\mathbb{R}^{n}\to\mathcal{C}$
denote the Euclidean projection onto the standard isotonic cone. That is, for any $\bm{x}\in\mathbb{R}^{n}$,
\begin{equation}
\mathcal{P}_{\mathcal{C}}\left(\bm{x}\right)\coloneqq\underset{\bm{y}\in\mathcal{C}}{\arg\min}\left\Vert \bm{x}-\bm{y}\right\Vert^{2}.\label{eq:projection-C}
\end{equation}

Next, we introduce two types of majorization.

\begin{definition}[\citet{marshall1979inequalities}]\label{defn:majorization}For any two vectors
$\bm{a},\bm{b}\in\mathbb{R}^{n}$, we say that $\bm{a}$ majorizes
$\bm{b}$, denoted $\bm{a}\succeq \bm{b}$, if
\[
\sum_{i=1}^{k}a_{[i]}\geq\sum_{i=1}^{k}b_{[i]}
\]
holds for all $1\leq k\leq n$, with equality when $k = n$.
\end{definition}

\begin{definition}[\cite{su2022truthful}]\label{defn:no-majorization}For any two vectors
$\bm{a},\bm{b}\in\mathbb{R}^{n}$, we say that $\bm{a}$ majorizes
$\bm{b}$ in the natural order, denoted $\bm{a}\succeq_{\mathrm{no}}\bm{b}$,
if
\[
\sum_{i=1}^{k}a_{i}\geq\sum_{i=1}^{k}b_{i}
\]
holds for all $1\leq k\leq n$, with equality when $k = n$.
\end{definition}

The following two lemmas show how the two majorization definitions relate to each other and showcase the connection between the first definition and the convex utility \eqref{eq:util_form} in Assumption~\ref{ass:convex}.

\begin{lemma} \label{lemma:no-majorization-projection}
For any two vectors $\bm{a},\bm{b}\in\mathbb{R}^{n}$, if $\bm{a}\succeq_{\mathrm{no}}\bm{b}$,
then $\mathcal{P}_{\mathcal{C}}(\bm{a})\succeq \mathcal{P}_{\mathcal{C}}(\bm{b})$.
\end{lemma}
\begin{proof}
	See \citet[Lemma 6.5]{su2022truthful}.
\end{proof}

\begin{lemma}[Hardy--Littlewood--P\'olya inequality] \label{lemma:hardy-littlewood}For any two vectors
$\bm{a},\bm{b}\in\mathbb{R}^{n}$, the inequality 
\[
\sum_{i=1}^{n}h\left(a_{i}\right)\geq\sum_{i=1}^{n}h\left(b_{i}\right)
\]
holds for any convex function $h:\mathbb{R}\to\mathbb{R}$
if and only if $\bm{a}\succeq \bm{b}$. \end{lemma}

\begin{proof}See \citet[Section 4.B.2]{marshall1979inequalities}.\end{proof}

\subsection{Proof of Theorem \ref{thm:truth-telling}}

Without loss of generality, we assume that the ground-truth ranking $\pi^{\star}(i)=i$. That is,
\[
\mu_{1}^{\star}\geq\mu_{2}^{\star}\geq\cdots\geq\mu_{n}^{\star}\qquad\text{and}\qquad\theta_{1}^\star \geq\theta_{2}^\star \geq\cdots\geq\theta_{n}^\star.
\]
For two different rankings $\pi$ and $\nu$, we say that $\pi$ is
an upward swap of $\nu$ if there exist $1\leq i<j\leq n$ such that
\[
\pi\left(i\right)=\nu\left(j\right)<\pi\left(j\right)=\nu\left(i\right)
\]
and $\pi(k)=\nu(k)$ for all $k\neq i,j$. Roughly speaking, $\pi$
and $\nu$ are identical except for their $i$th and $j$th entries, and
$\pi$ is ``close'' to $\pi^{\star}$. 

Our proof is a stone's throw once we show that
\begin{equation}
\mathsf{Util}\,(\pi)\geq\mathsf{Util}\,(\nu)\label{eq:improve}
\end{equation}
as long as $\pi$ is an upward swap of $\nu$. To see this point, note the following well-known fact concerning permutations:
\begin{lemma}\label{lemma:upward-swap}
For any permutation $\pi$ of $\{1, \ldots, n\}$, there exist an integer $m\geq1$ and a sequence of permutations
$\pi_{1},\ldots,\pi_{m}$ such that $\pi_{1}=\pi^{\star}$, $\pi_{m}=\pi$,
and for each $1\leq i\leq m-1$, $\pi_{i}$ is an upward swap of $\pi_{i+1}$. 

\end{lemma}

\begin{remark}
For completeness, a proof of this fact is given in the appendix.  
\end{remark}

To prove Theorem \ref{thm:truth-telling}, we can get
\[
\mathsf{Util}\,(\pi^{\star})=\mathsf{Util}\,(\pi_{1})\geq\mathsf{Util}\,(\pi_{2})\geq\cdots\geq\mathsf{Util}\,(\pi_{m})=\mathsf{Util}\,(\pi),
\]
by applying (\ref{eq:improve}) repeatedly, where $\{\pi_{i}\}_{1\leq i\leq m}$ are constructed in Lemma \ref{lemma:upward-swap}
that satisfies: $\pi_{1}=\pi^{\star}$, $\pi_{m}=\pi$, and for each
$1\leq i\leq m-1$, $\pi_{i}$ is an upward swap of $\pi_{i+1}$.

The remainder of the proof aims to establish \eqref{eq:improve}. To begin with, write
\[
\mathsf{Util}\,(\pi)=\mathbb{E}\left[\sum_{k=1}^{n}U\left(\left[\mathcal{P}_{\mathcal{C}}\left(\pi\circ \bX\right)\right]_{k}\right)\right]
\]
and
\[
\mathsf{Util}\,(\nu)=\mathbb{E}\left[\sum_{k=1}^{n}U\left(\left[\mathcal{P}_{\mathcal{C}}\left(\nu\circ \bX\right)\right]_{k}\right)\right],
\]
where $\mathcal{C}$ and $\mathcal{P}_{\mathcal{C}}$ are defined in (\ref{eq:C-defn}) and (\ref{eq:projection-C}), respectively. Next, define
\[
\Delta\left(x,y\right)=\mathbb{E}\left[\sum_{k=1}^{n}U\left(\left[\mathcal{P}_{\mathcal{C}}\left(\pi\circ \bX\right)\right]_{k}\right)-\sum_{k=1}^{n}U\left(\left[\mathcal{P}_{\mathcal{C}}\left(\nu\circ \bX\right)\right]_{k}\right)\mymid X_{\pi(i)}=x,X_{\pi(j)}=y\right].
\]
It is straightforward to check the following properties for this function:
\begin{enumerate}
\item[(a)] Recall that $\pi(i)=\nu(j)<\pi(j)=\nu(i)$ and $\pi(k)=\nu(k)$ for all $k\neq i,j$. Therefore, we have 
\begin{equation}
\Delta\left(x,y\right)=-\Delta\left(y,x\right)\label{eq:Delta-symmetric}
\end{equation}
for any $x,y$. This
immediately implies that
\begin{equation}
\Delta\left(x,x\right)=0.\label{eq:Delta-x-x-0}
\end{equation}
\item[(b)] If $x>y$, conditional on $X_{\pi(i)}=x$ and $X_{\pi(j)}=y$, we have $\pi\circ \bX\succeq_{\mathrm{no}}\nu\circ \bX$. By Lemma~\ref{lemma:no-majorization-projection}, we get
\[
\mathcal{P}_{\mathcal{C}}\left(\pi\circ \bX\right)\succeq\mathcal{P}_{\mathcal{C}}\left(\nu\circ X\right).
\]
Together with Lemma \ref{lemma:hardy-littlewood}, the Hardy--Littlewood--P\'olya inequality, the majorization relationship above yields
\[
\sum_{i=1}^{n}U\left(\left[\mathcal{P}_{\mathcal{C}}\left(\pi\circ X\right)\right]_{i}\right)\geq\sum_{i=1}^{n}U\left(\left[\mathcal{P}_{\mathcal{C}}\left(\nu\circ X\right)\right]_{i}\right)
\]
for any (nondecreasing) convex function $U$. Therefore we know that
\begin{equation}\label{eq:Delta-positive-negative}
\begin{aligned}
  \Delta(x,y) &\geq 0 \quad \text{for } x > y,  \\
  \Delta(x,y) &\leq 0 \quad \text{for } x < y.
\end{aligned}
\end{equation}

\end{enumerate}

Taken together, the elements above show that
\begin{align*}
\mathsf{Util}\,(\pi)-\mathsf{Util}\,(\nu) & \overset{\text{(i)}}{=}\mathbb{E}\left[\mathbb{E}\left[\sum_{k=1}^{n}U\left(\left[\mathcal{P}_{\mathcal{C}}\left(\pi\circ X\right)\right]_{k}\right)-\sum_{k=1}^{n}U\left(\left[\mathcal{P}_{\mathcal{C}}\left(\nu\circ X\right)\right]_{k}\right) \mymid X_{\pi(i)},X_{\pi(j)}\right]\right]\\
 & =\mathbb{E}\left[\Delta\left(X_{\pi(i)},X_{\pi(j)}\right)\right]\\
& = \int\Delta\left(x,y\right)P_{\theta_{\pi(i)}}\left(\mathrm{d}x\right)P_{\theta_{\pi(j)}}\left(\mathrm{d}y\right),
\end{align*}
where (i) indicates the use of the tower property of conditional expectation and $P_{\theta}(\mathrm{d}x)$ denotes the probability measure induced by the exponential family distribution \eqref{eq:exp}. To proceed, we have
\begin{align*}
&\int\Delta\left(x,y\right)P_{\theta_{\pi(i)}}\left(\mathrm{d}x\right)P_{\theta_{\pi(j)}}\left(\mathrm{d}y\right)\\
&\overset{\text{(ii)}}{=}\int_{x>y}\Delta\left(x,y\right)P_{\theta_{\pi(i)}}\left(\mathrm{d}x\right)P_{\theta_{\pi(j)}}\left(\mathrm{d}y\right)+\int_{x<y}\Delta\left(x,y\right)P_{\theta_{\pi(i)}}\left(\mathrm{d}x\right)P_{\theta_{\pi(j)}}\left(\mathrm{d}y\right)\\
 & \overset{\text{(iii)}}{=}\int_{x>y}\Delta\left(x,y\right)P_{\theta_{\pi(i)}}\left(\mathrm{d}x\right)P_{\theta_{\pi(j)}}\left(\mathrm{d}y\right)-\int_{x>y}\Delta\left(x,y\right)P_{\theta_{\pi(j)}}\left(\mathrm{d}x\right)P_{\theta_{\pi(i)}}\left(\mathrm{d}y\right)\\
 & =\int_{x>y}\Delta\left(x,y\right)c\left(x\right)c\left(y\right)\exp\left(\theta_{\pi(i)}x+\theta_{\pi(j)}y-b\left(\theta_{\pi(i)}\right)-b\left(\theta_{\pi(j)}\right)\right)\rho\left(\mathrm{d}x\right)\rho\left(\mathrm{d}y\right)\\
 & \quad-\int_{x>y}\Delta\left(x,y\right)c\left(x\right)c\left(y\right)\exp\left(\theta_{\pi(i)}y+\theta_{\pi(j)}x-b\left(\theta_{\pi(i)}\right)-b\left(\theta_{\pi(j)}\right)\right)\rho\left(\mathrm{d}x\right)\rho\left(\mathrm{d}y\right)\\
 & =\mathrm{e}^{-b\left(\theta_{\pi(i)}\right)-b\left(\theta_{\pi(j)}\right)}\int_{x>y}\Delta\left(x,y\right)c\left(x\right)c\left(y\right)\left[\mathrm{e}^{\theta_{\pi(i)}x+\theta_{\pi(j)}y}-\mathrm{e}^{\theta_{\pi(i)}y+\theta_{\pi(j)}x}\right]\rho\left(\mathrm{d}x\right)\rho\left(\mathrm{d}y\right)\\
 & \overset{\text{(iv)}}{\text{\ensuremath{\geq}}}0,
\end{align*}
where $\rho$ denotes the base measure with respect to the exponential family distribution.

To be complete, (ii) follows from (\ref{eq:Delta-x-x-0}); (iii) utilizes (\ref{eq:Delta-symmetric});
and (iv) holds since when $x>y$, we have $\Delta(x,y)\geq0$, which is ensured by \eqref{eq:Delta-positive-negative}, and that $\theta_{\pi(i)}x+\theta_{\pi(j)}y\geq\theta_{\pi(i)}y+\theta_{\pi(j)}x$ as $\theta_{\pi(i)}\geq\theta_{\pi(j)}$. Therefore,  we have proved (\ref{eq:improve}) when $\pi$ is an upward swap of $\nu$.

This concludes the proof of Theorem \ref{thm:truth-telling}.

\subsection{Proof of Theorem \ref{thm:coarse-optimality}}

Without loss of generality, as earlier, we assume that $\pi^{\star}(i)=i$:
\[
\mu_{1}^{\star}\geq\mu_{2}^{\star}\geq\cdots\geq\mu_{n}^{\star}\qquad\text{and}\qquad\theta_{1}^{\star}\geq\theta_{2}^{\star}\geq\cdots\geq\theta_{n}^{\star}.
\]
For any coarse ranking $\bm{I}\coloneqq(I_{1},\ldots,I_{p})$ of sizes
$n_{1},\ldots,n_{p}$, let $\pi_{\bm{I},\bm{X}}$ be the permutation
such that
\[
I_{1}=\left\{ \pi_{\bm{I},\bm{X}}\left(i\right)\right\} _{1\leq i\leq n_{1}},\quad I_{2}=\left\{ \pi_{\bm{I},\bm{X}}\left(i\right)\right\} _{n_{1}<i\leq n_{1}+n_{2}},\ldots,\quad I_{p}=\left\{ \pi_{\bm{I},\bm{X}}\left(i\right)\right\} _{n_{1}+\cdots+n_{p-1}<i\leq n}
\]
and 
\begin{equation}\label{eq:coarse-proof-sort}
X_{\pi_{\bm{I},\bm{X}}(i)}\geq X_{\pi_{\bm{I},\bm{X}}(j)}
\end{equation}
for all $i < j \in I_{k}, 1\le k \le p$. As is clear, $\pi_{\bm{I},\bm{X}}$ is a random permutation.

Henceforth, we omit the subscript $\bm{X}$ and use $\pi_{\bm{I}}$ to denote this permutation as long as it is clear from the context. It follows from \citet[Lemma 6.14]{su2022truthful} that the solution to the coarse Isotonic Mechanism \eqref{eq:isotonic-regression-coarse} with the coarse ranking $\bm{I}$ is equivalent to the Isotonic Mechanism with ranking $\pi=\pi_{\bm{I}}$. 

We say that a coarse ranking $\bm{I}$ is a coarse upward swap
of another coarse ranking $\bm{J}\coloneqq(J_{1},\ldots,J_{p})$ of
the same sizes if there exist $1\leq k_{1}<k_{2}\leq p$ and $a\in I_{k_{1}}$,
$b\in I_{k_{2}}$ such that 
\[
J_{k_{1}}=\left(I_{k_{1}}\setminus\left\{ a\right\} \right)\cup\left\{ b\right\} ,\quad J_{k_{2}}=\left(I_{k_{2}}\setminus\left\{ b\right\} \right)\cup\left\{ a\right\} ,\quad\text{and}\quad a<b.
\]
 Now, we will show that if $\bm{I}$ is a coarse upward swap of $\bm{J}$,
then 
\begin{equation}
\mathsf{Util}\,(\pi_{\bm{I}})\geq\mathsf{Util}\,(\pi_{\bm{J}}).\label{eq:coarse-proof-0}
\end{equation}

Let $j=\pi_{\bm{J}}^{-1}(a)$ and $i=\pi_{\bm{J}}^{-1}(b)$. By definition, $i<j$. Let $\pi_{\bm{J}\to\bm{I}}$ be derived from $\pi_{\bm{J}}$ by changing only the indices $i$ and $j$. That is,
\[
\pi_{\bm{J}\to\bm{I}}\left(i\right)=a,\quad\pi_{\bm{J}\to\bm{I}}\left(j\right)=b,\quad\text{and}\quad\pi_{\bm{J}\to\bm{I}}\left(l\right)=\pi_{\bm{J}}\left(l\right)
\]
for all $l\neq i,j$. Then, $\pi_{\bm{J}\to\bm{I}}$ is an upward swap of $\pi_{\bm{J}}$, which allows us to use (\ref{eq:improve}) to conclude that
\begin{equation}
\mathsf{Util}\,(\pi_{\bm{J}\to\bm{I}})\geq\mathsf{Util}\,(\pi_{\bm{J}}).\label{eq:coarse-proof-1}
\end{equation}
In addition, it is straightforward to check that 
\begin{align*}
& I_{1}=\left\{ \pi_{\bm{I}}\left(i\right)\right\} _{1\leq i\leq n_{1}}=\left\{ \pi_{\bm{J}\to\bm{I}}\left(i\right)\right\} _{1\leq i\leq n_{1}},\quad \ldots \\
& I_{p}=\left\{ \pi_{\bm{I}}\left(i\right)\right\} _{n_{1}+\cdots+n_{p-1}<i\leq n}=\left\{ \pi_{\bm{J}\to\bm{I}}\left(i\right)\right\} _{n_{1}+\cdots+n_{p-1}<i\leq n},
\end{align*}
which together with (\ref{eq:coarse-proof-sort}) establishes
\[
\pi_{\bm{I}}\circ\bm{X}\succeq_{\mathrm{no}}\pi_{\bm{J}\to\bm{I}}\circ\bm{X}.
\]
By applying Lemma~\ref{lemma:no-majorization-projection}, this immediately gives
\[
\mathcal{P}_{\mathcal{C}}\left(\pi_{\bm{I}}\circ\bm{X}\right)\succeq \mathcal{P}_{\mathcal{C}}\left(\pi_{\bm{J}\to\bm{I}}\circ\bm{X}\right).
\]
In view of the Hardy--Littlewood--P\'olya inequality (Lemma \ref{lemma:hardy-littlewood}), we get
\[
\sum_{i=1}^{n}U\left(\left[\mathcal{P}_{\mathcal{C}}\left(\pi_{\bm{I}}\circ\bm{X}\right)\right]_{i}\right)\geq\sum_{i=1}^{n}U\left(\left[\mathcal{P}_{\mathcal{C}}\left(\pi_{\bm{J}\to\bm{I}}\circ\bm{X}\right)\right]_{i}\right)
\]
for any (nondecreasing) convex function $U$, which immediately gives
\begin{equation}
\mathsf{Util}\,(\pi_{\bm{I}})\geq\mathsf{Util}\,(\pi_{\bm{J}\to\bm{I}}).\label{eq:coarse-proof-2}
\end{equation}
Taking (\ref{eq:coarse-proof-1}) and (\ref{eq:coarse-proof-2}) collectively
proves (\ref{eq:coarse-proof-0}).

As with Lemma \ref{lemma:upward-swap} for rankings, for any coarse ranking $\bm{I}$ that is not truthful, we can show that there exist
a sequence of coarse rankings $\bm{I}_{1},\ldots,\bm{I}_{m}$ of the same sizes such that $\bm{I}_{1}=\bm{I}^{\star}$ (the truthful coarse
ranking), $\bm{I}_{m}=\bm{I}$, and for each $1\leq j\leq m-1$,
$\bm{I}_{j}$ is a coarse upward swap of $\bm{I}_{j+1}$. This fact allows us to apply (\ref{eq:coarse-proof-0}) repeatedly and obtain
\[
\mathsf{Util}\,(\pi_{\bm{I}^{\star}})=\mathsf{Util}\,(\pi_{\bm{I}_{1}})\geq\mathsf{Util}\,(\pi_{\bm{I}_{2}})\geq\cdots\geq\mathsf{Util}\,(\pi_{\bm{I}_{m}})=\mathsf{Util}\,(\pi_{\bm{I}}),
\]
thereby finishing the proof.


%% file: hyperplane.tex
\section{Ranking is (almost) optimal for truthfulness}
\label{sec:rank-almost-optim}
In this section, we present the proof for Theorem \ref{thm:necessary-condition-slope} and Proposition \ref{thm:necessary-condition-intercept}. We begin by gathering some preliminary facts in Section \ref{sec:necessary-prelim} and then outline the proof in Section \ref{sec:proof-architecture}. The proofs of Theorem \ref{thm:necessary-condition-slope} and Proposition \ref{thm:necessary-condition-intercept} are provided in Sections \ref{sec:proof-slope} and \ref{sec:proof-intercept}, respectively.


\subsection{Preliminaries} \label{sec:necessary-prelim}

\begin{definition}[\citet{marshall1979inequalities}]\label{defn:w-majorization}For any two vectors
$\bm{a},\bm{b}\in\mathbb{R}^{n}$, we say that $\bm{a}$ weakly majorizes $\bm{b}$, denoted $\bm{a}\succeq_{\mathrm{w}} \bm{b}$, if
\[
\sum_{i=1}^{k}a_{[i]}\geq\sum_{i=1}^{k}b_{[i]}
\]
holds for all $1\leq k\leq n$.
\end{definition}

\begin{lemma}[Hardy--Littlewood--P\'olya inequality] \label{lemma:hardy-littlewood2}
For any two vectors
$\bm{a},\bm{b}\in\mathbb{R}^{n}$, the inequality 
\[
\sum_{i=1}^{n}h\left(a_{i}\right)\geq\sum_{i=1}^{n}h\left(b_{i}\right)
\]
holds for all nondecreasing convex function $h:\mathbb{R}\to\mathbb{R}$
if and only if $\bm{a}\succeq_{\mathrm{w}} \bm{b}$. \end{lemma}

\begin{proof}See \citet[Section 4.B.2]{marshall1979inequalities}.\end{proof}

Suppose that the linear knowledge partition in $\mathbb{R}^{n}$ is constructed
by a finite set of separating hyperplanes, where a hyperplane is defined
as $\{\bm{x}:\bm{a}^{\top}\bm{x}=b\}$ for some $\bm{a}\in\mathbb{R}^{n}$
and $b\in\mathbb{R}$. We group these hyperplanes into $K$ groups
$\mathcal{H}_{1},\ldots,\mathcal{H}_{K}$, where each group $\mathcal{H}_{k}$
consists of $l_{k}$ parallel hyperplanes. To be precise, we let
\[
\mathcal{H}_{k}=\left\{ H_{k,1},\ldots,H_{k,l_{k}}\right\} ,\quad\text{where}\quad H_{k,l}=\left\{ \bm{x}:\bm{a}_{k}^{\top}\bm{x}=b_{k,l}\right\} ,\quad\forall\,1\leq k\leq K,\,1\leq l\leq l_{k}.
\]
In addition, we assume that $\Vert\bm{a}_{k}\Vert_{2}=1$ and $b_{k,1}>\cdots>b_{k,l_{k}}$
for all $1\leq k\leq K$ for identifiability. We assume that the hyperplanes
in different groups are not parallel, namely for each $1\leq k_{1}\neq k_{2}\leq K$,
$\bm{a}_{k_{1}}\neq\bm{a}_{k_{2}}$. We also define a set of hyperplanes
\[
I_{i,j}\coloneqq\left\{ \bm{x}:x_{i}-x_{j}=0\right\} ,\qquad\forall\,1\leq i\neq j\leq n.
\]
We also collect a few basic facts in geometry, which will be useful
throughout the proof. For any point $\bm{x}_{0}\in\mathbb{R}^{n}$
and a hyperplane $H=\{\bm{x}:\bm{a}^{\top}\bm{x}=b\}$, the Euclidean
projection of $\bm{x}_{0}$ onto $H$ is
\begin{equation}
\mathcal{P}_{H}\left(\bm{x}_{0}\right)=\bm{x}_{0}-\frac{\bm{a}^{\top}\bm{x}_{0}-b}{\left\Vert \bm{a}\right\Vert^{2}}\bm{a}.\label{eq:euclidean-projection}
\end{equation}
The Euclidean distance between $\bm{x}_{0}$ and $H$ is
\begin{equation}
\mathsf{dist}\left(\bm{x}_{0},H\right)=\frac{\left|\bm{a}^{\top}\bm{x}_{0}-b\right|}{\left\Vert \bm{a}\right\Vert}.\label{eq:distance-to-hyperplane}
\end{equation}

\subsection{Proof architecture}
\label{sec:proof-architecture}

In this section we present the outline of the proof. For simplicity
we only focus on the case where each group $\mathcal{H}_{k}$ only
contains one hyperplane (denoted by $H_{k}$), and the review scores $X_i\sim\mathcal{N}(\mu_i^\star,\sigma^2)$. The general setting
will be treated in the formal proof. 

We first show that each hyperplane $H_{k}$ is parallel to some $I_{i,j}$.
In order to better appreciate our proof idea, we first sketch the
proof of \cite[Theorem 1]{su2022truthful}, which can be viewed as
the noiseless case ($\sigma=0$). That paper proves the result by contradiction:
if $H_{k}$ is not parallel to any $I_{i,j}$, then we will be able
to find a point $\bm{\omega}\in H_{k}$ such that each entry of $\bm{\omega}$
is different. For some sufficiently small $\varepsilon>0$, assuming
that $\bm{\omega}+\varepsilon\bm{a}_{k}$ is the true score, then
$\bm{\omega}+\varepsilon\bm{a}_{k}$ should weakly majorize its projection
onto the knowledge element on the other side of $H_{k}$, which is
given by $\bm{\omega}$; similarly one can also deduce that $\bm{\omega}-\varepsilon\bm{a}_{k}$
also weakly majorizes $\bm{\omega}$. In view of \cite[Lemma 6.3]{su2022truthful},
this leads to a contradiction. However in the noisy case ($\sigma>0$),
the analysis becomes challenging since the noisy score can take value
in the entire $\mathbb{R}^{n}$, and there is no simple expression
for the projection of the noisy score onto other knowledge element.
This issue can be resolved by moving $\bm{\omega}$ and $\bm{\omega}\pm\varepsilon\bm{a}_{k}$
far away from the origin and changing the scale, which can effectively be
viewed as reducing the magnitude of noise. More concretely, we can
find a vector $\bm{u}$ that is parallel to $H_{k}$ and has distinct
entries. Then for any $t>0$, we can consider two possible ground
truths $\bm{\omega}+t\bm{u}\pm\varepsilon t\bm{a}_{k}$, and the noisy
scores look like $\bm{\omega}+t\bm{u}\pm\varepsilon t\bm{a}_{k}+\mathcal{N}(\bm{0},\sigma^{2}\bm{I}_{n})$.
By rescaling everything with $1/t$, as $t$ goes to infinity we can
show that both $\bm{u}\pm\varepsilon\bm{a}_{k}$ weakly majorize $\bm{u}$,
which (similar to the noiseless case) leads to a contradiction. 

Then we show that each hyperplane $H_{k}$ passes through the origin. If this
is not true, we will find a point $\bm{y}\in H_{k}$ that is sufficiently
far away from the origin and other hyperplanes $\{H_{k'}:k'\neq k\}$.
Consider the case when the ground truth is $\bm{y}$ and the utility
function takes the form $U(\bm{x})=\Vert\bm{x}\Vert_{2}^{2}$ , then
we can compute the expected utility when one reports the knowledge
element on both sides of $H_{k}$ that contains $\bm{y}$ respectively,
which turns out to be different. This leads to a contradiction, which
in turn shows that each hyperplane must pass through the origin. Taking the
above two results collectively shows that each hyperplane $H_{k}$
must be a pairwise comparison hyperplane.

\subsection{Step 1: determining the normal vectors $\{\bm{a}_{k}\}_{1\protect\leq k\protect\leq K}$ (Proof of Theorem \ref{thm:necessary-condition-slope})} \label{sec:proof-slope}

In Step 1, we demonstrate that for any given $1\leq k\leq K$, there must
exists $1\leq i<j\leq n$ such that $\bm{a}_{k}$ is a constant multiple
of the vector $\bm{e}_{i}-\bm{e}_{j}$. We will establish this through a proof by contradiction.
Suppose this statement is false; then, we can identify a direction $\bm{u}\in\mathbb{R}^{d}$ satisfying: (i) $\bm{u}$ is parallel to the hyperplanes in $\mathcal{H}_{k}$, (ii) $\bm{u}$ is not parallel to the hyperplanes in other $\mathcal{H}_{k'}$ for $k'\neq k$, and (iii) $\bm{u}$ is not parallel to the hyperplane $I_{i,j}$ for any $1\leq i<j\leq n$. To see why this is true, we can express these conditions equivalently using the following systems of equations/inequalities:
\begin{equation}
\begin{cases}
\bm{a}_{k}^{\top}\bm{u}=0,\\
\bm{a}_{k'}^{\top}\bm{u}\neq0, & \text{for all }k'\neq k\\
u_{i}\neq u_{j}, & \text{for all }1\leq i\neq j\leq n.
\end{cases}\label{eq:equation-system-1}
\end{equation}
The existence of such $\bm{u}$ is evident given that $\bm{a}_{k'}\neq\bm{a}_{k}$
for all $k'\neq k$ and that $\bm{a}_{k}$ is not a constant multiple
of $\bm{e}_{i}-\bm{e}_{j}$ for all $1\leq i\neq j\leq n$. We can then select an arbitrary $\bm{u}$ satisfying (\ref{eq:equation-system-1}) and assume $\Vert\bm{u}\Vert=1$ for identifiability.

Next, we fix an arbitrary point $\bm{y}_{1}$ on $H_{k,1}$, and let
$\bm{y}_{2}$ be its Euclidean projection onto $H_{k,l_{k}}$. Define
\[
\bm{y}_{1}\left(t\right)\coloneqq\bm{y}_{1}+t\bm{u}\qquad\text{and}\qquad\bm{y}_{2}\left(t\right)\coloneqq\bm{y}_{2}+t\bm{u}.
\]
We know that 
\begin{equation}
\mathsf{dist}\left(\bm{y}_{1}\left(t\right),\bm{y}_{2}\left(t\right)\right)=\frac{b_{k,1}-b_{k,l_{k}}}{\left\Vert \bm{a}_{k}\right\Vert }=b_{k,1}-b_{k,l_{k}},\qquad\forall\,t\geq0\label{eq:y1t-y2t-dist}
\end{equation}
We can use (\ref{eq:distance-to-hyperplane}) to compute, for $s=1,2$
\begin{align*}
\mathsf{dist}\left(\bm{y}_{s}\left(t\right),H_{k',l}\right) & =\left|\bm{a}_{k'}^{\top}\left(\bm{y}_{s}+t\bm{u}\right)-b_{k',l}\right|=\left|\bm{a}_{k'}^{\top}\bm{y}_{s}-b_{k',l}+t\bm{a}_{k'}^{\top}\bm{u}\right|\\
 & \geq t\left|\bm{a}_{k'}^{\top}\bm{u}\right|-\left|\bm{a}_{k'}^{\top}\bm{y}_{s}-b_{k',l}\right|
\end{align*}
for any $k'\neq k$ and $1\leq l\leq l_{k'}$. In view of (\ref{eq:equation-system-1}),
there exists a constant
\[
c_{1}\coloneqq\frac{1}{4}\min_{k'\neq k}\left|\bm{a}_{k'}^{\top}\bm{u}\right|>0
\]
and a sufficiently large constant $T_{1}$ satifsfying
\[
T_{1}>\frac{1}{2c_{1}}\max_{k'\neq k,s\in\{1,2\},1\leq l\leq l_{k'}}\left|\bm{a}_{k'}^{\top}\bm{y}_{s}-b_{k',l}\right|,
\]
such that for any $t\geq T_{1}$ and $s=1,2$
\begin{equation}
\mathsf{dist}\left(\bm{y}_{s}\left(t\right),H_{k',l}\right)\geq2c_{1}t\label{eq:dist-y-H-step1}
\end{equation}
holds for any $k'\neq k$, $1\leq l\leq l_{k'}$.

Define 
\[
\bm{\omega}_{1}\coloneqq\bm{y}_{1}\left(T_{1}\right)\qquad\text{and}\qquad\bm{\omega}_{2}=\bm{y}_{2}\left(T_{1}\right).
\]
For any $t\geq0$ and $\varepsilon,\delta>0$, we further define
\begin{align*}
\bm{\omega}_{1}\left(t\right) & \coloneqq\bm{y}_{1}\left(T_{1}+t\right)=\bm{\omega}_{1}+t\bm{u},\\
\bm{\omega}_{2}\left(t\right) & \coloneqq\bm{y}_{2}\left(T_{1}+t\right)=\bm{\omega}_{2}+t\bm{u},
\end{align*}
and
\begin{align*}
\bm{z}_{1}\left(t,\varepsilon\right) & \coloneqq\bm{y}_{1}\left(T_{1}+t\right)+\varepsilon t\bm{a}_{k}=\bm{\omega}_{1}\left(t\right)+\varepsilon t\bm{a}_{k}=\bm{\omega}_{1}+t\bm{u}+\varepsilon t\bm{a}_{k},\\
\bm{z}_{2}\left(t,\varepsilon\right) & \coloneqq\bm{y}_{2}\left(T_{1}+t\right)-\varepsilon t\bm{a}_{k}=\bm{\omega}_{2}\left(t\right)-\varepsilon t\bm{a}_{k}=\bm{\omega}_{2}+t\bm{u}-\varepsilon t\bm{a}_{k},
\end{align*}
as well as
\begin{align*}
\widetilde{\bm{z}}_{1}\left(t,\delta\right) & \coloneqq\bm{y}_{1}\left(T_{1}+t\right)-\delta\bm{a}_{k}=\bm{\omega}_{1}\left(t\right)-\delta\bm{a}_{k}=\bm{\omega}_{1}+t\bm{u}-\delta\bm{a}_{k},\\
\widetilde{\bm{z}}_{2}\left(t,\delta\right) & \coloneqq\bm{y}_{2}\left(T_{1}+t\right)+\delta\bm{a}_{k}=\bm{\omega}_{2}\left(t\right)+\delta\bm{a}_{k}=\bm{\omega}_{2}+t\bm{u}+\delta\bm{a}_{k}.
\end{align*}
It is easy to check that $\bm{z}_{1}(t,\varepsilon)$, $\bm{\omega}_{1}(t)$,
$\bm{\omega}_{2}(t)$ and $\bm{z}_{1}(t,\varepsilon)$ lie on a line
sequentially, and the Euclidean projection of $\bm{z}_{1}(t,\varepsilon)$
(resp.~$\bm{z}_{2}(t,\varepsilon)$) onto $H_{k,1}$ (resp.~$H_{k,l_{k}}$)
is $\bm{\omega}_{1}(t)$ (resp.~$\bm{\omega}_{2}(t)$). Fix any constant
$\varepsilon_{0}\leq c_{1}$. For any $t\geq0$ and any $0<\varepsilon\leq\varepsilon_{0}$,
we can show that the line segment between $\bm{\omega}_{1}(t)$ (resp.~$\bm{\omega}_{2}(t)$)
and $\bm{z}_{1}(t,\varepsilon)$ (resp.~$\bm{z}_{2}(t,\varepsilon)$)
does not cross any separating hyperplane $H_{k,l}$ (except that its
endpoint $\bm{\omega}_{1}(t)$ (resp.~$\bm{\omega}_{2}(t)$) is on
$H_{k,1}$ (resp.~$H_{k,l_{k}}$). To see why this is true, for $s=1,2$,
we have $\mathsf{dist}(\bm{z}_{s}(t,\varepsilon),\bm{\omega}_{s}(t))=\varepsilon t$,
and from (\ref{eq:dist-y-H-step1}) we know that
\[
\mathsf{dist}\left(\bm{\omega}_{s}\left(t\right),H_{k',l}\right)=\mathsf{dist}\left(\bm{\omega}_{s}\left(T_{1}+t\right),H_{k',l}\right)\geq2c_{1}\left(T_{1}+t\right)
\]
for any $k'\neq k$ and $1\leq l\leq l_{k'}$. Therefore the two line
segments does not cross any $H_{k',l}$ for $k'\neq k$ and $1\leq l\leq l_{k'}$,
as long as $\varepsilon\leq\varepsilon_{0}\leq c_{1}$. Recall that
we assume $b_{k,1}>\cdots>b_{k,l_{k}}$, and as a result the two line
segments does not cross any $H_{k,l}$ for $1<l<l_{k}$. This gives
two immediate consequences: 
\begin{enumerate}
\item For any $t\geq0$ and any $0<\varepsilon\leq\varepsilon_{0}$ we have
\[
\mathsf{dist}\left(\bm{z}_{s}\left(t,\varepsilon\right),H_{k',l}\right)\geq\mathsf{dist}\left(\bm{\omega}_{s}\left(t\right),H_{k',l}\right)-\mathsf{dist}\left(\bm{\omega}_{s}\left(t\right),\bm{z}_{s}\left(t,\varepsilon\right)\right)\geq2c_{1}T+c_{1}t
\]
for any $k'\neq k$ and $1\leq l\leq l_{k'}$, and
\[
\mathsf{dist}\left(\bm{z}_{s}\left(t,\varepsilon\right),H_{k,l}\right)\geq\varepsilon t
\]
for any $1\leq l\leq l_{k}$. Taking the above two inequalities collectively
yields
\begin{equation}
\min_{1\leq k\leq K,1\leq l\leq l_{k}}\mathsf{dist}\left(\bm{z}_{s}\left(t,\varepsilon\right),H_{k,l}\right)\geq\min\left\{ c_{1},\varepsilon\right\} t=\varepsilon t\label{eq:dist-z-H}
\end{equation}
as long as $\varepsilon\leq\varepsilon_{0}\leq c_{1}$. 
\item As an immediate consequence of (\ref{eq:dist-z-H}), for $s=1,2$,
the set $\{\bm{z}_{1}(t,\varepsilon):t\geq0,0<\varepsilon\leq\varepsilon_{0}\}$
(resp.~$\{\bm{z}_{2}(t,\varepsilon):t\geq0,0<\varepsilon\leq\varepsilon_{0}\}$)
is a subset of an element $S_{1}$ (resp.~$S_{2}$) from the knowledge
partition. This is because the only intersection of this set with
the separating hyperplanes $\{H_{k,l}:1\leq k\leq K,1\leq l\leq l_{k}\}$
is $\bm{z}_{1}(0,0)=\bm{\omega}_{1}$ (resp.~$\bm{z}_{2}(0,0)=\bm{\omega}_{2}$). 
\end{enumerate}
By similar arguments, we can also show that there exists a constant
\[
\delta_{0}\coloneqq\begin{cases}
\frac{1}{2}\min\left\{ \frac{b_{k,1}-b_{k,2}}{\left\Vert \bm{a}_{k}\right\Vert _{2}},\frac{b_{k,l_{k}-1}-b_{k,l_{k}}}{\left\Vert \bm{a}_{k}\right\Vert _{2}},c_{1}T\right\} , & \text{if }l_{k}\geq2,\\
\frac{1}{2}cT, & \text{if }l_{k}=1,
\end{cases}
\]
such that for $s=1,2$, the set $\{\widetilde{\bm{z}}_{1}(t,\delta):t\geq0,0<\delta\leq\delta_{0}\}$
(resp.~$\{\widetilde{\bm{z}}_{2}(t,\delta):t\geq0,0<\varepsilon\leq\delta_{0}\}$)
is a subset of an element $S_{1}'$ (resp.~$S_{2}'$) from the linear knowledge
partition. 

Fix any $0<\varepsilon\leq\varepsilon_{0}$. For any $t\geq0$, when
the ground truth is $\bm{z}_{1}(t,\varepsilon)$, let the noisy score be $\bX_t$. By the truth telling
assumption we know that for any nondecreasing convex function $U:\mathbb{R}\to\mathbb{R}$
we have
\begin{equation}
\mathbb{E}\left[U\left(\mathcal{P}_{S_{1}}\left(\bm{z}_{1}\left(t,\varepsilon\right)+\bm{e}\right)\right)\right]\geq\mathbb{E}\left[U\left(\bm{\omega}_{1}\left(t\right)+\mathcal{P}_{H_{k,1}}\left(\bm{e}\right)\right)\right],\label{eq:U-expectation}
\end{equation}
where $\bm{e}=\bX_t-\bz_1(t,\varepsilon)$ is the
noise, and $\mathcal{P}_{S_{1}}(\bm{z}_{1}(t,\varepsilon)+\bm{e})$
(resp.~$\bm{\omega}_{1}(t)+\mathcal{P}_{H_{k,1}}(\bm{e})$) is the
Euclidean projection of the noisy score $\bX_t$
onto $S_{1}$ (resp.~$S_{1}'$). Then we claim that this leads to
$\bm{u}+\varepsilon\bm{a}_{k}\succeq_{\mathrm{w}}\bm{u}$. Actually
if this does not hold, then Lemma \ref{lemma:hardy-littlewood} asserts
the existence of a nondecreasing continuous convex function $h:\mathbb{R}\to\mathbb{R}$
such that
\begin{equation}
h\left(\bm{u}+\varepsilon\bm{a}_{k}\right)<h\left(\bm{u}\right).\label{eq:h-contradiction}
\end{equation}
Let $U(x)=h(tx)$, which is also a nondecreasing convex function.
Then (\ref{eq:U-expectation}) gives
\begin{equation}
\mathbb{E}\left[h\left(\frac{\mathcal{P}_{S_{1}}\left(\bm{z}_{1}\left(t,\varepsilon\right)+\bm{e}\right)}{t}\right)\right]\geq\mathbb{E}\left[h\left(\frac{\bm{\omega}_{1}\left(t\right)+\mathcal{P}_{H_{k,1}}\left(\bm{e}\right)}{t}\right)\right].\label{eq:h-expectation}
\end{equation}
In view of the assumption \eqref{eq:growth-condition}, we know that 
\[
\frac{\bm{e}}{t} = \frac{\bX_t-\bz_1(t,\varepsilon)}{t}\overset{\mathrm{a.s.}}{\longrightarrow}\bm{0}
\]
We also learn from (\ref{eq:dist-z-H}) that the distance between $\bm{z}_{1}(t,\varepsilon)$
and the boundry of $S_{1}$ (which must be one of the separating hyperplanes)
is at least $\varepsilon t$, therefore as $t\to\infty$ the sequence
of random vector
\[
\frac{\left(\mathcal{I}-\mathcal{P}_{S_{1}}\right)\left(\bm{z}_{1}\left(t,\varepsilon\right)+\bm{e}\right)}{t}\overset{\mathrm{a.s.}}{\longrightarrow}\bm{0},
\]
where $\mathcal{I}$ is the identity operator. This further gives
\begin{align*}
\frac{\mathcal{P}_{S_{1}}\left(\bm{z}_{1}\left(t,\varepsilon\right)+\bm{e}\right)}{t} & =\frac{\bm{z}_{1}\left(t,\varepsilon\right)+\bm{e}}{t}-\frac{\left(\mathcal{I}-\mathcal{P}_{S_{1}}\right)\left(\bm{z}_{1}\left(t,\varepsilon\right)+\bm{e}\right)}{t}\\
 & =\frac{\bm{\omega}_{1}+t\bm{u}+\varepsilon t\bm{a}_{k}+\bm{e}}{t}-\frac{\left(\mathcal{I}-\mathcal{P}_{S_{1}}\right)\left(\bm{z}_{1}\left(t,\varepsilon\right)+\bm{e}\right)}{t}\\
 & \overset{\mathrm{a.s.}}{\longrightarrow}\bm{u}+\varepsilon\bm{a}_{k}
\end{align*}
as $t\to\infty$. It is also straightforward to check that
\[
\frac{\bm{\omega}_{1}\left(t\right)+\mathcal{P}_{H_{k,1}}\left(\bm{e}\right)}{t}=\frac{\bm{\omega}_{1}+t\bm{u}+\mathcal{P}_{H_{k,1}}\left(\bm{e}\right)}{t}\overset{\mathrm{a.s.}}{\longrightarrow}\bm{u}.
\]
as $t\to\infty$. Since $h(\cdot)$ is continuous, by the continuous
mapping theorem, we know that 
\[
\mathbb{E}\left[h\left(\frac{\bm{z}_{1}\left(t,\varepsilon\right)+\bm{e}}{t}\right)\right]\to h\left(\bm{u}+\varepsilon\bm{a}_{k}\right)\qquad\text{and}\qquad\mathbb{E}\left[h\left(\frac{\bm{\omega}_{1}\left(t\right)+\mathcal{P}_{H_{k,1}}\left(\bm{e}\right)}{t}\right)\right]\to h\left(\bm{u}\right),
\]
as $t\to\infty$. Since (\ref{eq:h-expectation}) holds for any $t$,
we can take the limit $t\to\infty$ to achieve $h(\bm{u}+\varepsilon\bm{a}_{k})\geq h(\bm{u})$,
which contradicts with (\ref{eq:h-contradiction}). This shows that
$\bm{u}+\varepsilon\bm{a}_{k}\succeq_{\mathrm{w}}\bm{u}$. Similarly
we can also show that $\bm{u}-\varepsilon\bm{a}_{k}\succeq_{\mathrm{w}}\bm{u}$.
Since each element of $\bm{u}$ is different, and $\bm{u}\pm\varepsilon\bm{a}_{k}\succeq_{\mathrm{w}}\bm{u}$
both hold for any $\varepsilon\in(0,\varepsilon_{0}]$, we can apply
\cite[Lemma 6.3]{su2022truthful} to show that $\varepsilon=0$, which
is a contradiction. This means that for any given $1\leq k\leq K$,
there must exist $1\leq i<j\leq n$ such that $\bm{a}_{k}$ is a
constant multiple of the vector $\bm{e}_{i}-\bm{e}_{j}$, which finishes
Step 1 as well as the proof of Theorem \ref{thm:necessary-condition-slope}.

\subsection{Step 2: determining the intercepts $\{b_{k,l}\}_{1\protect\leq k\protect\leq K,1\protect\leq l\protect\leq l_{k}}$ (Proof of Proposition \ref{thm:necessary-condition-intercept})} \label{sec:proof-intercept}

Based on what we have shown in Step 1, we can further assume, without
loss of generality, that for any $1\leq k\leq K$
\[
\bm{a}_{k}=\frac{1}{\sqrt{2}}\left(\bm{e}_{i_{k}}-\bm{e}_{j_{k}}\right)
\]
for some $1\leq i_{k}<j_{k}\leq n$. We also require that $(i_{k},j_{k})\neq(i_{k'},j_{k'})$
for $k\neq k'$. In Step 2 we aim to show that for any $1\leq k\leq K$,
we must have $l_{k}=1$. We prove this by contradiction. Suppose there
exists $1\leq k\leq K$ such that $l_{k}\geq2$, and assume without
loss of generality that $b_{k,1}>0$ (the case when $l_{k}=1$ is
much easier, and the proof can be obtained by taking $b_{k,2}=-\infty$
in the following proof).

For any $t\geq0$, let $\bm{v}$ be a unit vector satisfying
\begin{equation}
\begin{cases}
\bm{a}_{k}^{\top}\bm{v}=0,\\
\bm{a}_{k'}^{\top}\bm{v}\neq0, & \text{for all }k'\neq k,\\
v_{i}>0, & \text{for all }1\leq i\leq n.
\end{cases}\label{eq:equation-system-2}
\end{equation}
Fix any $\bm{y}(0)\in H_{k,1}$, and define
\[
\bm{y}\left(t\right)=\bm{y}\left(0\right)+t\bm{v},\qquad\forall\,t\geq0.
\]
Note that for any $1\leq l\leq l_{k}$ and any $t\geq0$, we have
\[
\mathsf{dist}\left(\bm{y}\left(t\right),H_{k,l}\right)=\frac{b_{k,1}-b_{k,l}}{\left\Vert \bm{a}_{k}\right\Vert _{2}}=b_{k,1}-b_{k,l},
\]
Akin to Step 1, we can check that $\bm{y}_{t}\in H_{k,1}$ and there
exists some sufficiently large constant $T>0$ and some sufficiently
small constant $c_{2}>0$, such that for any $t\geq T$, we have
\begin{equation}
\min_{k'\neq k,1\leq l\leq l_{k'}}\mathsf{dist}\left(\bm{y}\left(t\right),H_{k',l}\right)\geq4c_{2}t,\label{eq:dist-y-H-step2}
\end{equation}
and
\begin{equation}
\max_{1\leq l\leq l_{k}}\mathsf{dist}\left(\bm{y}\left(t\right),H_{k,l}\right)\leq c_{2}t,\label{eq:dist-y-H-k-2}
\end{equation}
and
\begin{equation}
\min_{1\leq i\leq n}y_{i}\left(t\right)\geq4c_{2}t.\label{eq:y-entrywise}
\end{equation}
Define a Euclidean ball 
\[
\mathsf{B}\left(t\right)\coloneqq\left\{ \bm{x}\in\mathbb{R}^{n}:\left\Vert \bm{x}-\bm{y}\left(t\right)\right\Vert _{2}\leq2c_{2}t\right\} ,
\]
and its three subsets separated by $H_{k,1}$ and $H_{k,2}$: 
\begin{align*}
\mathsf{B}_{1}\left(t\right) & \coloneqq\mathsf{B}\left(t\right)\cap\left\{ \bm{x}:\bm{a}_{k}^{\top}\bm{x}\geq b_{k,1}\right\} ,\\
\mathsf{B}_{2}\left(t\right) & \coloneqq\mathsf{B}\left(t\right)\cap\left\{ \bm{x}:b_{k,1}\geq\bm{a}_{k}^{\top}\bm{x}\geq b_{k,2}\right\} ,\\
\mathsf{B}_{3}\left(t\right) & \coloneqq\mathsf{B}\left(t\right)\cap\left\{ \bm{x}:\bm{a}_{k}^{\top}\bm{x}\leq b_{k,2}\right\} .
\end{align*}
It is straightforward to check that for all $t\geq T$, $\mathsf{B}_{1}(t)$
(resp.~$\mathsf{B}_{2}(t)$) lives in the same element, denoted by
$S_{1}$ (resp.~$S_{2}$), of the linear knowledge partition. For any $t\geq T$,
when the ground truth is $\bm{y}(t)$, by the truth-telling assumption
we know that for any nondecreasing convex function $U:\mathbb{R}\to\mathbb{R}$
we have
\begin{equation}
\mathbb{E}\left[U\left(\mathcal{P}_{S_{1}}\left(\bm{y}\left(t\right)+\bm{e}\right)\right)\right]=\mathbb{E}\left[U\left(\mathcal{P}_{S_{2}}\left(\bm{y}\left(t\right)+\bm{e}\right)\right)\right]\label{eq:U-expectation-1}
\end{equation}
with $\bm{e}\sim\mathcal{N}(\bm{0},\sigma^{2}\bm{I}_{n})$, because $\bm{y}(t)\in S_{1}$ and $\bm{y}(t)\in S_{2}$ hold
simultaneously. We take a test function of the form $U(x)=\max\{x^{2},0\}$,
which is nondecreasing and convex. We first decompose both sides
of (\ref{eq:U-expectation-1}) into
\begin{align*}
\mathbb{E}\left[U\left(\mathcal{P}_{S_{1}}\left(\bm{y}\left(t\right)+\bm{e}\right)\right)\right] & =\underbrace{\mathbb{E}\left[U\left(\mathcal{P}_{S_{1}}\left(\bm{y}\left(t\right)+\bm{e}\right)\right)\ind\left\{ \bm{y}\left(t\right)+\bm{e}\in\mathsf{B}_{1}\left(t\right)\right\} \right]}_{\eqqcolon\alpha_{1}\left(t\right)}\\
 & \quad+\underbrace{\mathbb{E}\left[U\left(\mathcal{P}_{S_{1}}\left(\bm{y}\left(t\right)+\bm{e}\right)\right)\ind\left\{ \bm{y}\left(t\right)+\bm{e}\in\mathsf{B}_{2}\left(t\right)\cup\mathsf{B}_{3}\left(t\right)\right\} \right]}_{\eqqcolon\alpha_{2}\left(t\right)}\\
 & \quad+\underbrace{\mathbb{E}\left[U\left(\mathcal{P}_{S_{1}}\left(\bm{y}\left(t\right)+\bm{e}\right)\right)\ind\left\{ \bm{y}\left(t\right)+\bm{e}\notin\mathsf{B}\left(t\right)\right\} \right]}_{\eqqcolon\alpha_{3}\left(t\right)}
\end{align*}
and
\begin{align*}
\mathbb{E}\left[U\left(\mathcal{P}_{S_{2}}\left(\bm{y}\left(t\right)+\bm{e}\right)\right)\right] & =\underbrace{\mathbb{E}\left[U\left(\mathcal{P}_{S_{2}}\left(\bm{y}\left(t\right)+\bm{e}\right)\right)\ind\left\{ \bm{y}\left(t\right)+\bm{e}\in\mathsf{B}_{1}\left(t\right)\right\} \right]}_{\eqqcolon\beta_{1}\left(t\right)}\\
 & \quad+\underbrace{\mathbb{E}\left[U\left(\mathcal{P}_{S_{2}}\left(\bm{y}\left(t\right)+\bm{e}\right)\right)\ind\left\{ \bm{y}\left(t\right)+\bm{e}\in\mathsf{B}_{2}\left(t\right)\right\} \right]}_{\eqqcolon\beta_{2}\left(t\right)}\\
 & \quad+\underbrace{\mathbb{E}\left[U\left(\mathcal{P}_{S_{2}}\left(\bm{y}\left(t\right)+\bm{e}\right)\right)\ind\left\{ \bm{y}\left(t\right)+\bm{e}\in\mathsf{B}_{3}\left(t\right)\right\} \right]}_{\eqqcolon\beta_{3}\left(t\right)}\\
 & \quad+\underbrace{\mathbb{E}\left[U\left(\mathcal{P}_{S_{2}}\left(\bm{y}\left(t\right)+\bm{e}\right)\right)\ind\left\{ \bm{y}\left(t\right)+\bm{e}\notin\mathsf{B}\left(t\right)\right\} \right]}_{\eqqcolon\beta_{4}\left(t\right)}.
\end{align*}
Then we analyze $\alpha_{1}(t)$, $\alpha_{2}(t)$, $\alpha_{3}(t)$,
$\beta_{1}(t)$, $\beta_{2}(t)$, $\beta_{3}(t)$, and $\beta_{4}(t)$,
respectively.
\begin{itemize}
\item In view of (\ref{eq:dist-y-H-step2}) and (\ref{eq:dist-y-H-k-2}),
we know that when $\bm{y}(t)+\bm{e}\in\mathsf{B}_{1}(t)$,
\begin{equation}
\mathcal{P}_{S_{2}}\left(\bm{y}\left(t\right)+\bm{e}\right)=\mathcal{P}_{H_{k,1}}\left(\bm{y}\left(t\right)+\bm{e}\right)=\bm{y}\left(t\right)+\mathcal{P}_{H_{k,1}}\left(\bm{e}\right),\label{eq:step2-proj-1}
\end{equation}
when $\bm{y}(t)+\bm{e}\in\mathsf{B}_{2}(t)\cup\mathsf{B}_{3}(t)$,
\begin{equation}
\mathcal{P}_{S_{1}}\left(\bm{y}\left(t\right)+\bm{e}\right)=\mathcal{P}_{H_{k,1}}\left(\bm{y}\left(t\right)+\bm{e}\right)=\bm{y}\left(t\right)+\mathcal{P}_{H_{k,1}}\left(\bm{e}\right),\label{eq:step2-proj-2}
\end{equation}
By symmetry of $\bm{e}\sim\mathcal{N}(\bm{0},\sigma^{2}\bm{I}_{n})$,
we know that 
\begin{equation}
\alpha_{2}\left(t\right)=\beta_{1}\left(t\right).\label{eq:alpha-2-beta-1}
\end{equation}
\item In view of (\ref{eq:y-entrywise}), we know that for any $\bm{x}\in\mathsf{B}(t)$,
$\bm{x}$ has positive entries and therefore $U(\bm{x})=\Vert\bm{x}\Vert_{2}^{2}$.
Let $\phi(\bm{x})=(2\pi)^{-n/2}\exp(-\Vert\bm{x}\Vert_{2}^{2}/2)$
be the probability density function of $\bm{e}\sim\mathcal{N}(\bm{0},\sigma^{2}\bm{I}_{n})$.
Recall that $\mathsf{B}_{1}(t)\subset S_{1}$, $\mathsf{B}_{2}(t)\subset S_{2}$,
we know that 
\begin{align*}
\alpha_{1}\left(t\right) & =\mathbb{E}\left[U\left(\bm{y}\left(t\right)+\bm{e}\right)\ind\left\{ \bm{y}\left(t\right)+\bm{e}\in\mathsf{B}_{1}\left(t\right)\right\} \right]=\mathbb{E}\left[\left\Vert \bm{y}\left(t\right)+\bm{e}\right\Vert _{2}^{2}\ind\left\{ \bm{y}\left(t\right)+\bm{e}\in\mathsf{B}_{1}\left(t\right)\right\} \right]\\
 & =\int_{\bm{a}_{k}^{\top}\bm{x}\geq0,\left\Vert \bm{x}\right\Vert _{2}\leq2c_{2}t}\left\Vert \bm{y}\left(t\right)+\bm{x}\right\Vert _{2}^{2}\phi\left(\bm{x}\right)\mathrm{d}\bm{x}=\int_{\bm{a}_{k}^{\top}\bm{x}\leq0,\left\Vert \bm{x}\right\Vert _{2}\leq2c_{2}t}\left\Vert \bm{y}\left(t\right)-\bm{x}\right\Vert _{2}^{2}\phi\left(\bm{x}\right)\mathrm{d}\bm{x},
\end{align*}
where the last line holds since $\phi(\cdot)$ is symmetric around
$\bm{0}$, and
\begin{align*}
\beta_{2}\left(t\right) & =\mathbb{E}\left[U\left(\bm{y}\left(t\right)+\bm{e}\right)\ind\left\{ \bm{y}\left(t\right)+\bm{e}\in\mathsf{B}_{2}\left(t\right)\right\} \right]=\mathbb{E}\left[\left\Vert \bm{y}\left(t\right)+\bm{e}\right\Vert _{2}^{2}\ind\left\{ \bm{y}\left(t\right)+\bm{e}\in\mathsf{B}_{2}\left(t\right)\right\} \right]\\
 & =\int_{0\geq\bm{a}_{k}^{\top}\bm{x}\geq b_{k,2}-b_{k,1},\left\Vert \bm{x}\right\Vert _{2}\leq2c_{2}t}\left\Vert \bm{y}\left(t\right)+\bm{x}\right\Vert _{2}^{2}\phi\left(\bm{x}\right)\mathrm{d}\bm{x}
\end{align*}
In view of (\ref{eq:dist-y-H-step2}) and (\ref{eq:dist-y-H-k-2}),
we know that when $\bm{y}(t)+\bm{e}\in\mathsf{B}_{3}(t)$,
\begin{align*}
\mathcal{P}_{S_{2}}\left(\bm{y}\left(t\right)+\bm{e}\right) & =\mathcal{P}_{H_{k,2}}\left(\bm{y}\left(t\right)+\bm{e}\right)=\bm{y}\left(t\right)+\bm{e}-\left\{ \bm{a}_{k}^{\top}\left[\bm{y}\left(t\right)+\bm{e}\right]-b_{k,2}\right\} \bm{a}_{k}=\bm{y}\left(t\right)+\bm{e}-c_{t}\left(\bm{e}\right)\bm{a}_{k},
\end{align*}
where we define
\[
c_{t}\left(\bm{x}\right)\coloneqq\bm{a}_{k}^{\top}\left[\bm{y}\left(t\right)+\bm{x}\right]-b_{k,2}=\bm{a}_{k}^{\top}\bm{x}+b_{k,1}-b_{k,2}.
\]
Therefore we know that
\begin{align*}
\beta_{3}\left(t\right) & =\mathbb{E}\left[\left\Vert \bm{y}\left(t\right)+\bm{e}-c_{t}\left(\bm{e}\right)\bm{a}_{k}\right\Vert _{2}^{2}\ind\left\{ \bm{y}\left(t\right)+\bm{e}\in\mathsf{B}_{3}\left(t\right)\right\} \right]\\
 & =\int_{\bm{a}_{k}^{\top}\bm{x}\leq b_{k,2}-b_{k,1},\left\Vert \bm{x}\right\Vert _{2}\leq2c_{2}t}\left\Vert \bm{y}\left(t\right)+\bm{x}-c_{t}\left(\bm{x}\right)\bm{a}_{k}\right\Vert _{2}^{2}\phi\left(\bm{x}\right)\mathrm{d}\bm{x}.
\end{align*}
\item Then we have
\begin{align*}
\alpha_{1}\left(t\right)-\beta_{2}\left(t\right)-\beta_{3}\left(t\right) & =\underbrace{\int_{0\geq\bm{a}_{k}^{\top}\bm{x}\geq b_{k,2}-b_{k,1},\left\Vert \bm{x}\right\Vert _{2}\leq2c_{2}t}\left[\left\Vert \bm{y}\left(t\right)-\bm{x}\right\Vert _{2}^{2}-\left\Vert \bm{y}\left(t\right)+\bm{x}\right\Vert _{2}^{2}\right]\phi\left(\bm{x}\right)\mathrm{d}\bm{x}}_{\eqqcolon\gamma_{1}\left(t\right)}\\
 & +\underbrace{\int_{\bm{a}_{k}^{\top}\bm{x}\leq b_{k,2}-b_{k,1},\left\Vert \bm{x}\right\Vert _{2}\leq2c_{2}t}\left[\left\Vert \bm{y}\left(t\right)-\bm{x}\right\Vert _{2}^{2}-\left\Vert \bm{y}\left(t\right)+\bm{x}-c_{t}\left(\bm{x}\right)\bm{a}_{k}\right\Vert _{2}^{2}\right]\phi\left(\bm{x}\right)\mathrm{d}\bm{x}}_{\eqqcolon\gamma_{2}\left(t\right)}.
\end{align*}
Regarding $\gamma_{1}(t)$, we have
\begin{align*}
\gamma_{1}\left(t\right) & =\int_{0\geq\bm{a}_{k}^{\top}\bm{x}\geq b_{k,2}-b_{k,1},\left\Vert \bm{x}\right\Vert _{2}\leq2c_{2}t}\left[-4\bm{y}\left(t\right)^{\top}\bm{x}\right]\phi\left(\bm{x}\right)\mathrm{d}\bm{x}\\
 & =\int_{0\geq\bm{a}_{k}^{\top}\bm{x}\geq b_{k,2}-b_{k,1},\left\Vert \bm{x}\right\Vert _{2}\leq2c_{2}t}\left[-4\bm{y}\left(0\right)^{\top}\bm{x}+t\bm{v}^{\top}\bm{x}\right]\phi\left(\bm{x}\right)\mathrm{d}\bm{x}\\
 & =\int_{0\geq\bm{a}_{k}^{\top}\bm{x}\geq b_{k,2}-b_{k,1},\left\Vert \bm{x}\right\Vert _{2}\leq2c_{2}t}\left[-4\bm{y}\left(0\right)^{\top}\bm{x}\right]\phi\left(\bm{x}\right)\mathrm{d}\bm{x}
\end{align*}
Here the penultimate step follows from $\bm{a}_{k}^{\top}\bm{v}=0$,
which is guaranteed by (\ref{eq:equation-system-2}). Similarly, we
also have
\begin{align*}
\gamma_{2}\left(t\right) & =\int_{\bm{a}_{k}^{\top}\bm{x}\leq b_{k,2}-b_{k,1},\left\Vert \bm{x}\right\Vert _{2}\leq2c_{2}t}\left\{ -4\bm{y}\left(t\right)^{\top}\bm{x}+2c_{t}\left(\bm{x}\right)\bm{a}_{k}^{\top}\left[\bm{y}\left(t\right)+\bm{x}\right]-c_{t}\left(\bm{x}\right)^{2}\right\} \phi\left(\bm{x}\right)\mathrm{d}\bm{x}\\
 & =\int_{\bm{a}_{k}^{\top}\bm{x}\leq b_{k,2}-b_{k,1},\left\Vert \bm{x}\right\Vert _{2}\leq2c_{2}t}\left\{ -4\bm{y}\left(0\right)^{\top}\bm{x}+2c_{t}\left(\bm{x}\right)\left(\bm{a}_{k}^{\top}\bm{x}+b_{k,1}\right)-c_{t}\left(\bm{x}\right)^{2}\right\} \phi\left(\bm{x}\right)\mathrm{d}\bm{x}\\
 & =\int_{\bm{a}_{k}^{\top}\bm{x}\leq b_{k,2}-b_{k,1},\left\Vert \bm{x}\right\Vert _{2}\leq2c_{2}t}\left[-4\bm{y}\left(0\right)^{\top}\bm{x}+2\left(\bm{a}_{k}^{\top}\bm{x}+b_{k,1}-b_{k,2}\right)b_{k,2}+\left(\bm{a}_{k}^{\top}\bm{x}+b_{k,1}-b_{k,2}\right)^{2}\right]\phi\left(\bm{x}\right)\mathrm{d}\bm{x}\\
 & =\int_{\bm{a}_{k}^{\top}\bm{x}\leq b_{k,2}-b_{k,1},\left\Vert \bm{x}\right\Vert _{2}\leq2c_{2}t}\left[-4\bm{y}\left(0\right)^{\top}\bm{x}+\left(\bm{a}_{k}^{\top}\bm{x}+b_{k,1}\right)^{2}-b_{k,2}^{2}\right]\phi\left(\bm{x}\right)\mathrm{d}\bm{x}.
\end{align*}
It is straightforward to check that
\begin{align*}
\lim_{t\to\infty}\gamma_{1}\left(t\right) & =\int_{0\geq\bm{a}_{k}^{\top}\bm{x}\geq b_{k,2}-b_{k,1}}\left[-4\bm{y}\left(0\right)^{\top}\bm{x}\right]\phi\left(\bm{x}\right)\mathrm{d}\bm{x}\eqqcolon\gamma_{1},\\
\lim_{t\to\infty}\gamma_{2}\left(t\right) & =\int_{\bm{a}_{k}^{\top}\bm{x}\leq b_{k,2}-b_{k,1}}\left[-4\bm{y}\left(0\right)^{\top}\bm{x}+\left(\bm{a}_{k}^{\top}\bm{x}+b_{k,1}\right)^{2}-b_{k,2}^{2}\right]\phi\left(\bm{x}\right)\mathrm{d}\bm{x}\eqqcolon\gamma_{2}.
\end{align*}
Let $Z_{1}=\bm{a}_{k}^{\top}\bm{e}$, $Z_{2}=\bm{v}^{\top}\bm{e}$,
and $Z_{3}=[(\bm{I}_{n}-\bm{a}_{k}\bm{a}_{k}^{\top})\bm{y}(0)]^{\top}\bm{e}$.
It is clear that $Z_{1}$ and $Z_{3}$ are independent. Recall that
$\bm{a}_{k}^{\top}\bm{v}=0$, we know that $Z_{1},Z_{2}\overset{\text{i.i.d.}}{\sim}\mathcal{N}(0,1)$.
We also know that
\[
\bm{y}\left(0\right)^{\top}\bm{e}=\left[\bm{a}_{k}\bm{a}_{k}^{\top}\bm{y}\left(0\right)+\left(\bm{I}_{n}-\bm{a}_{k}\bm{a}_{k}^{\top}\right)\bm{y}\left(0\right)\right]^{\top}\bm{e}=b_{k,1}Z_{1}+Z_{3}.
\]
Then we know that
\begin{align*}
\gamma_{1} & =\mathbb{E}\left[-4\left(b_{k,1}Z_{1}+Z_{3}\right)\ind\left\{ b_{k,2}-b_{k,1}\leq Z_{1}\leq0\right\} \right]\\
 & =-4b_{k,1}\mathbb{E}\left[Z_{1}\ind\left\{ b_{k,2}-b_{k,1}\leq Z_{1}\leq0\right\} \right],
\end{align*}
where the last step uses the fact that $Z_{1}$ and $Z_{3}$ are independent.
We also have
\begin{align*}
\gamma_{2} & =\mathbb{E}\left[\left[-4\left(b_{k,1}Z_{1}+Z_{3}\right)+\left(Z_{1}+b_{k,1}\right)^{2}-b_{k,2}^{2}\right]\ind\left\{ Z_{1}\leq b_{k,2}-b_{k,1}\right\} \right]\\
 & =\mathbb{E}\left[\left[-4b_{k,1}Z_{1}+\left(Z_{1}+b_{k,1}\right)^{2}-b_{k,2}^{2}\right]\ind\left\{ Z_{1}\leq b_{k,2}-b_{k,1}\right\} \right].
\end{align*}
Therefore 
\begin{align*}
\gamma_{1}+\gamma_{2} & =-4b_{k,1}\mathbb{E}\left[Z_{1}\ind\left\{ Z_{1}\leq0\right\} \right]+\mathbb{E}\left[\left[\left(Z_{1}+b_{k,1}\right)^{2}-b_{k,2}^{2}\right]\ind\left\{ Z_{1}\leq b_{k,2}-b_{k,1}\right\} \right]\\
 & =\frac{4}{\sqrt{2\pi}}b_{k,1}+\mathbb{E}\left[\left[\left(Z_{1}+b_{k,1}\right)^{2}-b_{k,2}^{2}\right]\ind\left\{ Z_{1}\leq b_{k,2}-b_{k,1}\right\} \right].
\end{align*}
Define a function 
\begin{align*}
f\left(x\right) & \coloneqq\mathbb{E}\left[\left[\left(Z_{1}+b_{k,1}\right)^{2}-x^{2}\right]\ind\left\{ Z_{1}\leq x-b_{k,1}\right\} \right]\\
 & =\int_{-\infty}^{x-b_{k,1}}\left[\left(z+b_{k,1}\right)^{2}-x^{2}\right]\frac{1}{\sqrt{2\pi}}\exp\left(-\frac{z^{2}}{2}\right)\mathrm{d}z,
\end{align*}
and compute its derivative
\begin{align*}
f'\left(x\right) & =\left[\left(x-b_{k,1}+b_{k,1}\right)^{2}-x^{2}\right]\frac{1}{\sqrt{2\pi}}\exp\left[-\frac{\left(x-b_{k,1}\right)^{2}}{2}\right]-\int_{-\infty}^{x-b_{k,1}}\frac{2x}{\sqrt{2\pi}}\exp\left(-\frac{z^{2}}{2}\right)\mathrm{d}z\\
 & =-x\int_{-\infty}^{x-b_{k,1}}\frac{2}{\sqrt{2\pi}}\exp\left(-\frac{z^{2}}{2}\right)\mathrm{d}z.
\end{align*}
We can see that $f(x)$ is increasing when $x\leq0$, and is decreasing
when $x>0$. As a result,
\begin{align*}
\inf_{x\leq b_{k,1}}f\left(x\right) & =\min\left\{ \inf_{x\to-\infty}f\left(x\right),f\left(b_{k,1}\right)\right\} =\min\left\{ 0,\int_{-\infty}^{0}\left(z^{2}+2b_{k,1}z\right)\frac{1}{\sqrt{2\pi}}\exp\left(-\frac{z^{2}}{2}\right)\mathrm{d}z\right\} \\
 & =\min\left\{ 0,\frac{1}{2}-\frac{2}{\sqrt{2\pi}}b_{k,1}\right\} .
\end{align*}
Therefore we reach
\begin{align*}
\gamma_{1}+\gamma_{2} & =\frac{4}{\sqrt{2\pi}}b_{k,1}+f\left(b_{k,2}\right)\geq\frac{4}{\sqrt{2\pi}}b_{k,1}+\inf_{x\leq b_{k,1}}f\left(x\right)=\frac{4}{\sqrt{2\pi}}b_{k,1}+\min\left\{ 0,\frac{1}{2}-\frac{2}{\sqrt{2\pi}}b_{k,1}\right\} \\
 & =\min\left\{ \frac{4}{\sqrt{2\pi}}b_{k,1},\frac{1}{2}+\frac{2}{\sqrt{2\pi}}b_{k,1}\right\} >0,
\end{align*}
namely we have
\begin{equation}
\lim_{t\to\infty}\alpha_{1}\left(t\right)-\beta_{2}\left(t\right)-\beta_{3}\left(t\right)=\gamma_{1}+\gamma_{2}>0.\label{eq:alpha-1-beta-2-beta-3}
\end{equation}
\item Lastly, we will analyze $\alpha_{3}(t)$ and $\beta_{4}(t)$, which
is relatively straightforward. We have
\begin{align*}
\alpha_{3}\left(t\right) & =\mathbb{E}\left[U\left(\mathcal{P}_{S_{1}}\left(\bm{y}\left(t\right)+\bm{e}\right)\right)\ind\left\{ \bm{y}\left(t\right)+\bm{e}\notin\mathsf{B}\left(t\right)\right\} \right]\\
 & =\mathbb{E}\left[U\left(\mathcal{P}_{S_{1}}\left(\bm{y}\left(t\right)+\bm{e}\right)\right)\ind\left\{ \left\Vert \bm{e}\right\Vert _{2}\geq2c_{2}t\right\} \right]\\
 & \leq\mathbb{E}\left[\left\Vert \mathcal{P}_{S_{1}}\left(\bm{y}\left(t\right)+\bm{e}\right)\right\Vert _{2}^{2}\ind\left\{ \left\Vert \bm{e}\right\Vert _{2}\geq2c_{2}t\right\} \right],
\end{align*}
where the last step follows from $U(x)=\max\{0,x^{2}\}\leq x^{2}$.
By the non-expansiveness of Euclidean projection, we have
\[
\left\Vert \mathcal{P}_{S_{1}}\left(\bm{y}\left(t\right)+\bm{e}\right)-\mathcal{P}_{S_{1}}\left(\bm{0}\right)\right\Vert _{2}\leq\left\Vert \bm{y}\left(t\right)+\bm{e}-\bm{0}\right\Vert _{2}=\left\Vert \bm{y}\left(t\right)+\bm{e}\right\Vert _{2}.
\]
Therefore we know that
\[
0\leq\alpha_{3}\left(t\right)\leq\mathbb{E}\left[\left(\left\Vert \bm{y}\left(t\right)+\bm{e}\right\Vert _{2}+\left\Vert \mathcal{P}_{S_{1}}\left(\bm{0}\right)\right\Vert _{2}\right)\ind\left\{ \left\Vert \bm{e}\right\Vert _{2}\geq2c_{2}t\right\} \right].
\]
By taking $t\to\infty$, we immediately have
\begin{equation}
\lim_{t\to\infty}\alpha_{3}\left(t\right)=0.\label{eq:alpha-3}
\end{equation}
Similarly we can also show that 
\begin{equation}
\lim_{t\to\infty}\beta_{4}\left(t\right)=0.\label{eq:beta-4}
\end{equation}
\end{itemize}
From (\ref{eq:U-expectation-1}) we know that 
\[
\alpha_{1}\left(t\right)+\alpha_{2}\left(t\right)+\alpha_{3}\left(t\right)-\beta_{1}\left(t\right)-\beta_{2}\left(t\right)-\beta_{3}\left(t\right)-\beta_{4}\left(t\right)=0
\]
for all $t\geq T$. But from (\ref{eq:alpha-2-beta-1}), (\ref{eq:alpha-1-beta-2-beta-3}),
(\ref{eq:alpha-3}) and (\ref{eq:beta-4}) we know that
\[
\lim_{t\to\infty}\alpha_{1}\left(t\right)+\alpha_{2}\left(t\right)+\alpha_{3}\left(t\right)-\beta_{1}\left(t\right)-\beta_{2}\left(t\right)-\beta_{3}\left(t\right)-\beta_{4}\left(t\right)>0,
\]
which is a contradiction. Therefore we have shown that $l_{k}=1$
and $b_{k,1}=0$ for all $1\leq k\leq K$, which finishes Step 2 and
the proof of Proposition \ref{thm:necessary-condition-intercept}.


%% file: proof_minimax.tex
\section{Minimax optimality} \label{sec:proof-minimax}

We now present a formal version of Theorem \ref{thm:general_minimax_informal}. Recall that for $i=1,\ldots,n$, the independent review scores $X_{i}\sim p_{\theta_{i}^{\star}}(x)$, where $p_{\theta}(x)=\mathrm{e}^{\theta x-b(\theta)}c(x)$. We have $\mu_{i}^{\star}=\mathbb{E}[X_{i}]=b'(\theta_{i}^{\star})$ and $\mathsf{Var}(X_{i})=b''(\theta_{i}^{\star})$. We assume the true scores lie within a closed interval $[V_{\min},V_{\max}]$, where $V_{\max}$ and $V_{\min}$ can be problem-specific, and define the parameter space
\[
\mathcal{I}_{n}\left(V_{\min},V_{\max}\right)\coloneqq\left\{ \bm{\mu}^{\star}\in\mathbb{R}^{n}:V_{\max}\geq\mu_{\pi^{\star}(1)}^{\star}\geq\cdots\geq\mu_{\pi^{\star}(n)}^{\star}\geq V_{\min}\right\} .
\]
Without loss of generality, we consider only positive scores, i.e., $V_{\min}\geq0$. As before, we assume that $b''(\theta)>0$ on its domain, which essentially requires that the noisy scores have nonzero variance. Then, we have $\theta_{\max}\geq\theta_{\pi^{\star}(1)}^{\star}\geq\cdots\geq\theta_{\pi^{\star}(n)}^{\star}\geq\theta_{\min}$, where $\theta_{\max}=(b')^{-1}(V_{\max})$ and $\theta_{\min}=(b')^{-1}(V_{\min})$. 
Let 
\[
\sigma^{2}\coloneqq\max_{\theta_{\min}\leq\theta\leq\theta_{\max}}b''\left(\theta\right)
\]
represent the largest variance of any random variable $X\sim p_{\theta}(x)$ for any $\theta\in[\theta_{\min},\theta_{\max}]$. We assume that $V_{\max}-V_{\min}\geq\sigma/\sqrt{n}$ to ensure the estimation problem is non-trivial. We proceed by assuming the following regularity condition.

\begin{assumption}\label{assumption:variance}There exist two 
	constants $C_{\mathsf{var}},C_{\mathsf{int}}>0$ such that, there
	exists $[\widetilde{V}_{\min},\widetilde{V}_{\max}]\subseteq[V_{\min},V_{\max}]$
	satisfying $\widetilde{V}_{\max}-\widetilde{V}_{\min}\geq C_{\mathsf{int}}(V_{\max}-V_{\min})$,
	such that 
	\[
	\mathsf{Var}_{X\sim p_{\theta}(x)}\left(X\right)=b''\left(\theta\right)\geq C_{\mathsf{var}}\sigma^{2}\quad\text{where}\quad\theta=\left(b'\right)^{-1}\left(\mu\right)
	\]
	for any $\mu\in[\widetilde{V}_{\min},\widetilde{V}_{\max}]$.
\end{assumption}

Assumption \ref{assumption:variance} is mild, as it essentially requires that the variance of the noisy score not change dramatically when the true score lies within a non-negligible part of the $[V_{\min},V_{\max}]$ interval. After presenting the following minimax-optimal estimation guarantees, we will verify that it is satisfied by common exponential family distributions.

\begin{theorem}[Formal version of Theorem \ref{thm:general_minimax_informal}]\label{thm:general_minimax} 
	Suppose that Assumption \ref{assumption:variance} holds. Then there exist constants $C_{1},C_{2}>0$ depending on $C_{\mathsf{var}}$ and $C_{\mathsf{int}}$ such that 
	\begin{equation}
		\sup_{V_{\min} \le \bmu^{\star} \le V_{\max}}\,\mathbb{E} \left\Vert \widehat{\bmu}-\bm{\mu}^{\star}\right\Vert _{2}^{2}\leq C_{1}n\sigma^{2}\min\left\{ 1,\left(\frac{V_{\max}-V_{\min}}{n\sigma}\right)^{2/3}+\frac{\log n}{n}\right\} ,\label{eq:general_minimax_upper}
	\end{equation}
	where $\widehat{\bm{\mu}}$ is the optimal solution to (\ref{eq:isotonic-regression-no-b}),
	and
	\begin{equation}
		\inf_{\widetilde{\bm{\mu}}} \sup_{V_{\min} \le \bmu^{\star} \le V_{\max}} \mathbb{E}\left\Vert \widetilde{\bm{\mu}}-\bm{\mu}^{\star}\right\Vert _{2}^{2}\geq C_{2}n\sigma^{2}\min\left\{ 1,\left(\frac{V_{\max}-V_{\min}}{n\sigma}\right)^{2/3}+\frac{1}{n}\right\} ,\label{eq:general_minimax_lower}
	\end{equation}
	where the infimum is taken over all estimators $\widetilde{\bm{\mu}}$
	for the true mean vector $\bm{\mu}^{\star}$ based on observed data
	$X_{1},\ldots,X_{n}$. 
	
\end{theorem}

Before delving into the proof, let us first discuss the implication of Theorem \ref{thm:general_minimax} when applied to the following commonly used exponential family distributions.

 \begin{enumerate}
	\item Gaussian distribution $\mathcal{N}(\mu,\sigma^{2})$ with fixed variance
	$\sigma^{2}$ is an exponential family distribution with $\theta=\mu/\sigma^{2}$
	and $b(\theta)=\sigma^{2}\theta^{2}/2$. It can be verified that Assumption
	\ref{assumption:variance} holds with $\widetilde{V}_{\max}=V_{\max}$,
	$\widetilde{V}_{\min}=V_{\min}$, $C_{\mathsf{int}}=1$ and $C_{\mathsf{var}}=1$.
	In comparison to the standard MLE (which simply uses $X_{i}$ to estimate
	$\mu_{i}^{\star}$ for each $i\in \{1,\ldots,n\}$), with an estimation error of
	$\mathbb{E}\Vert\bm{X}-\bm{\mu}^{\star}\Vert^{2}=n\sigma^{2}$,
	the estimation error of isotonic regression is
	\[
	O\left(\min\left\{ 1,\left(\frac{V_{\max}-V_{\min}}{n\sigma}\right)^{2/3}+\frac{\log n}{n}\right\} \right)
	\]
	times smaller, which is significantly better when $V_{\max}-V_{\min}\ll n\sigma$.
	\item Binomial distribution $\mathsf{Binom}(m,p)$ with success probability
	$p$ and fixed number of trials $m$ is an exponential family distribution
	with $\theta=\log[p/(1-p)]$, $b(\theta)=m\log[1+\exp(\theta)]$ and
	$\mu=mp$. Since the mean is upper bounded by $m$ in this setting,
	we consider $V_{\min}=0$ and $V_{\max}=m$. Then $\sigma^{2}=m/4$
	and it can be confirmed that Assumption \ref{assumption:variance} holds
	with, e.g.~$\widetilde{V}_{\max}=3m/4$, $\widetilde{V}_{\min}=m/4$,
	$C_{\mathsf{int}}=1/2$ and $C_{\mathsf{var}}=3/4$. We can check
	that for binomial distribution, both the estimation error bound (\ref{eq:general_minimax_upper})
	and the minimax lower bound (\ref{eq:general_minimax_lower}) scales
	as $\min\{mn,m^{4/3}n^{1/3}\}$, suggesting that isotonic regression
	is minimax-optimal (no need to exclude logarithmic factors). The estimation
	error of the Isotonic Mechanism is $O(\min\{1,(m/n^{2})^{1/3}\})$
	times smaller than the standard MLE (which scales as $mn$), which
	is much better when $n\gg\sqrt{m}$. 
	\item The Poisson distribution $\mathsf{Poisson}(\lambda)$ with rate $\lambda\in\mathbb{R}^{+}$ is an exponential family distribution characterized by $\theta=\log\lambda$, $b(\theta)=\exp(\theta)$, and $\mu=\lambda$. We can verify that $\sigma^{2}=V_{\max}$, and Assumption \ref{assumption:variance} is satisfied with $\widetilde{V}_{\max}=V_{\max}$, $\widetilde{V}_{\min}=\max{V_{\min},V_{\max}/2}$, $C_{\mathsf{int}}=1/2$, and $C_{\mathsf{var}}=1/2$. The estimation error for isotonic regression is
	\[
	O\left(\min\left\{ 1,\left(\frac{V_{\max}-V_{\min}}{n\sqrt{V_{\max}}}\right)^{2/3}+\frac{\log n}{n}\right\} \right)
	\]
	times smaller than the estimation error of the standard MLE, particularly when $V_{\max}-V_{\min}\ll n\sqrt{V_{\max}}$.
	\item The Gamma distribution $\mathsf{Gamma}(m,\beta)$ with probability density
	function
	\[
	f\left(x\right)=\frac{1}{\Gamma\left(m\right)\beta^{m}}x^{m-1}\exp\left(-\frac{x}{\beta}\right),\qquad\forall\,x\in\mathbb{R}^{+}
	\]
	is an exponential family distribution with $\theta=-1/\beta$, $b(\theta)=-m\log(-\theta)$
	and $\mu=m\beta$. We can check that $\sigma^{2}=V_{\max}^{2}/m$,
	and Assumption \ref{assumption:variance} holds with $\widetilde{V}_{\max}=V_{\max}$,
	$\widetilde{V}_{\min}=\max\{V_{\min},V_{\max}/2\}$, $C_{\mathsf{int}}=1/2$
	and $C_{\mathsf{var}}=1/4$. When $V_{\max}-V_{\min}\ll nV_{\max}$, the
	estimation error of isotonic regression is at least $O(\min\{1,(m/n^{2})^{1/3}\})$
	times smaller than the standard MLE, which is significantly better when $n\gg\sqrt{m}$.
\end{enumerate}

For the remainder of this section, we provide the proof of Theorem \ref{thm:general_minimax}.

\subsection{Establishing the upper bound (\ref{eq:general_minimax_upper})}

The upper bound follows from \citet[Theorem 2.2]{zhang2002risk}. We
first identify the quantities in \citet{zhang2002risk} as 
\[
y_{i}=X_{i},\qquad\varepsilon_{i}=X_{i}-\mathbb{E}\left[X_{i}\right]=X_{i}-\mu_{i}^{\star},
\]
and it is straightforward to check that $V=V_{\max}-V_{\min}$, 
\[
\mathbb{E}\left[\varepsilon_{i}\right]=0,\qquad\mathbb{E}\left[\varepsilon_{i}^{2}\right]=\mathsf{var}\left(X_{i}\right)\leq\sigma^{2}.
\]
Then we can apply \citet[Theorem 2.2]{zhang2002risk} to achieve 
\begin{align*}
	R_{n} & \lesssim\sigma\min\left\{ 1,\left[\left(\frac{V_{\max}-V_{\min}}{n\sigma}\right)^{2/3}+\frac{\log n}{n}\right]^{1/2}\right\} 
\end{align*}
and as a result 
\[
\mathbb{E}\left\Vert \bm{\mu}-\bm{\mu}^{\star}\right\Vert^{2}=nR_{n}^{2}\left(f\right)\leq C_{1}n\sigma^{2}\min\left\{ 1,\left(\frac{V_{\max}-V_{\min}}{n\sigma}\right)^{2/3}+\frac{\log n}{n}\right\} 
\]
for some universal constant $C_{1}>0$, which finishes the proof for
the upper bound (\ref{eq:general_minimax_upper}).

\subsection{Establishing the lower bound (\ref{eq:general_minimax_lower})}

We follow the proof idea of \citet[Corollary 5]{bellec2015sharp},
which focuses on the case when the observations are Gaussian, to establish
the minimax lower bound for general exponential family distributions.
It suffices to consider the true scores in the interval $[\widetilde{V}_{\min},\widetilde{V}_{\max}]$,
which is a subset of $[V_{\min},V_{\max}]$. This motivates one to
define 
\[
\widetilde{V}\coloneqq\widetilde{V}_{\max}-\widetilde{V}_{\min}\geq C_{\mathsf{int}}\left(V_{\max}-V_{\min}\right).
\]
We look at the following two cases separately. 

\paragraph{Case 1: when $n$ is not too small.}

We first consider the case when 
\begin{equation}
	\left(\frac{\widetilde{V}}{n\sigma}\right)^{2/3}\geq\frac{4}{n}.\label{eq:proof-minimax-0}
\end{equation}
We would like to show that 
\begin{equation}
	\inf_{\widehat{\bm{\mu}}}\sup_{\bm{\mu}^{\star}\in\mathcal{I}_{n}\left(V_{\min},V_{\max}\right)}\mathbb{E}\,\big\Vert\widehat{\bm{\mu}}-\bm{\mu}^{\star}\big\Vert^{2}\geq C_{2}n\sigma^{2}\min\left\{ 1,\left(\frac{V_{\max}-V_{\min}}{n\sigma}\right)^{2/3}\right\} \label{eq:proof-minimax-1}
\end{equation}
for some universal constant $C_{2}>0$. Let 
\begin{equation}
	k\coloneqq\min\left\{ \left\lfloor \left(\frac{n\widetilde{V}^{2}}{c^{2}\sigma^{2}}\right)^{1/3}\right\rfloor ,n\right\} \label{eq:proof-minimax-2}
\end{equation}
for some constant $c=C_{\mathsf{var}}/16$. We know from (\ref{eq:proof-minimax-0})
that $k\geq8$ as long as $n\geq8$. Without loss of generality, we
assume that $n$ is a multiple of $k$ for simplicity; the proof for
the general case is similar. In view of the Varshamov-Gilbert bound
\citep[Lemma 2.9]{tsybakov2008introduction}, there exists $\Omega\coloneqq\{\bm{\omega}^{(0)},\ldots,\bm{\omega}^{(M)}\}\subset\{0,1\}^{k}$
satisfying $M\geq2^{k/8}$, $\bm{\omega}^{(0)}=\bm{0}$ and 
\begin{equation}
	d_{\mathsf{H}}\left(\bm{\omega}^{(i)},\bm{\omega}^{(j)}\right)\geq\frac{k}{8},\qquad\forall\,0\leq i<j\leq M,\label{eq:proof-minimax-3}
\end{equation}
where $d_{\mathsf{H}}(\bm{x},\bm{y})$ is the Hamming distance $\sum_{i=1}^{n}\ind\{x_{i}\neq y_{i}\}$
between two vectors $\bm{x},\bm{y}\in\mathbb{R}^{n}$. Let
\begin{equation}
	\gamma=c\sqrt{\frac{\sigma^{2}k}{n}}.\label{eq:proof-minimax-3.5}
\end{equation}
Then for each $\bm{\omega}\in\Omega$, we define $\bm{\mu}^{\bm{\omega}}\in\mathbb{R}^{n}$
as 
\[
\mu_{i}^{\bm{\omega}}=\widetilde{V}_{\min}+\left\lfloor \frac{i-1}{n}k\right\rfloor \frac{1}{k}\widetilde{V}+\gamma\omega_{\lfloor\frac{i-1}{n}k\rfloor+1}
\]
for $1\leq i \leq n$. Note that in view of (\ref{eq:proof-minimax-2}), we have $\gamma\leq\widetilde{V}/k$,
and it is straightforward to check that
\[
\widetilde{V}_{\min}\leq\mu_{1}^{\bm{\omega}}\leq\mu_{i}^{\bm{\omega}}\leq\mu_{n}^{\bm{\omega}}\leq\widetilde{V}_{\max}
\]
holds for any $1\leq i\leq n$. Therefore for any $\bm{\omega}\in\Omega$
we have $\bm{\mu}^{\bm{\omega}}\in\mathcal{I}_{n}(V_{\min},V_{\max})$.
In addition, recall that $n$ is a multiple of $k$, therefore for
any $\bm{\omega},\bm{\omega}'\in\Omega$ 
\begin{align}
	\big\Vert\bm{\mu}^{\bm{\omega}}-\bm{\mu}^{\bm{\omega}'}\big\Vert^{2} & =\gamma^{2}\sum_{i=1}^{n}\big(\omega_{\lfloor\frac{i-1}{n}k\rfloor+1}-\omega_{\lfloor\frac{i-1}{n}k\rfloor+1}'\big)^{2}=\frac{\gamma^{2}n}{k}d_{\mathsf{H}}\left(\omega,\omega'\right)\overset{\text{(i)}}{\geq}\frac{\gamma^{2}n}{8}=\frac{c^{2}}{8}\sigma^{2}k\nonumber \\
	& =\frac{c^{2}}{8}\sigma^{2}\min\left\{ \left\lfloor \left(\frac{n\widetilde{V}^{2}}{c^{2}\sigma^{2}}\right)^{1/3}\right\rfloor ,n\right\} \nonumber \\
	& \overset{\text{(ii)}}{\geq}\frac{c^{2}}{8}C_{\mathsf{int}}^{2/3}\sigma^{2}\min\left\{ \left\lfloor \left(\frac{n\left(V_{\max}-V_{\min}\right)^{2}}{c^{2}\sigma^{2}}\right)^{1/3}\right\rfloor ,n\right\} \label{eq:proof-minimax-4}
\end{align}
as long as $\bm{\omega}\neq\bm{\omega}'$. Here (i) follows from (\ref{eq:proof-minimax-3}),
and (ii) holds since $\widetilde{V}\geq C_{\mathsf{int}}V$. 

Consider any $\mu_{1},\mu_{2}\in[\widetilde{V}_{\min},\widetilde{V}_{\max}]$,
and let $\theta_{1}=(b')^{-1}(\mu_{1})$ and $\theta_{2}=(b')^{-1}(\mu_{2})$.
Then the Kullback-Leibler (KL) divergence between two distributions
$p_{\theta_{1}}$ and $p_{\theta_{2}}$ satisfies

\begin{align}
	\mathsf{KL}\left(p_{\theta_{1}}\,\Vert\,p_{\theta_{2}}\right) & =\int\log\frac{p_{\theta_{1}}\left(x\right)}{p_{\theta_{2}}\left(x\right)}p_{\theta_{1}}\left(x\right)\mathrm{d}x=\int\left[\left(\theta_{1}-\theta_{2}\right)x-b\left(\theta_{1}\right)+b\left(\theta_{2}\right)\right]p_{\theta_{1}}\left(x\right)\mathrm{d}x\nonumber \\
	& =\left(\theta_{1}-\theta_{2}\right)b'\left(\theta_{1}\right)-b\left(\theta_{1}\right)+b\left(\theta_{2}\right)\leq\frac{1}{2}\max_{\theta_{\min}\leq\theta\leq\theta_{\max}}b''\left(\theta\right)\left(\theta_{1}-\theta_{2}\right)^{2}\nonumber \\
	& =\frac{\sigma^{2}}{2}\left(\theta_{1}-\theta_{2}\right)^{2}\leq\frac{\left(\mu_{1}-\mu_{2}\right)^{2}}{2C_{\mathsf{var}}^{2}\sigma^{2}},\label{eq:proof-minimax-5}
\end{align}
where the last relation follows from
\[
\left(\mu_{1}-\mu_{2}\right)^{2}=\left[b'\left(\theta_{1}\right)-b'\left(\theta_{2}\right)\right]^{2}\geq\min_{\theta_{\min}\leq\theta\leq\theta_{\max}}\left[b''\left(\theta\right)\right]^{2}\left(\theta_{1}-\theta_{2}\right)^{2}\geq C_{\mathsf{var}}^{2}\sigma^{4}\left(\theta_{1}-\theta_{2}\right)^{2},
\]
which is a direct consequence of Assumption \ref{assumption:variance}.
For each $\bm{\omega}\in\Omega$, let $P_{\bm{\omega}}$ be the data
distribution when $\bm{\mu}^{\star}=\bm{\mu}^{\bm{\omega}}$. Then
we have 
\begin{align}
	\mathsf{KL}\left(P_{\bm{\omega}}\,\Vert\,P_{\bm{0}}\right) & \overset{\text{(i)}}{=}\sum_{i=1}^{n}\mathsf{KL}\left(p_{\theta_{i}^{\bm{\omega}}}\,\Vert\,p_{\theta_{i}^{\bm{0}}}\right)\overset{\text{(ii)}}{\leq}\sum_{i=1}^{n}\frac{\left(\mu_{i}^{\bm{\omega}}-\mu_{i}^{0}\right)^{2}}{2C_{\mathsf{var}}^{2}\sigma^{2}}=\frac{1}{2C_{\mathsf{var}}^{2}\sigma^{2}}\left\Vert \bm{\mu}^{\bm{\omega}}-\bm{\mu}^{\bm{0}}\right\Vert ^{2}\nonumber \\
	& \overset{\text{(iii)}}{=}\frac{1}{2C_{\mathsf{var}}^{2}\sigma^{2}}\frac{\gamma^{2}n}{k}d_{\mathsf{H}}\left(\bm{\omega},\bm{0}\right)\overset{\text{(iv)}}{\leq}\frac{\gamma^{2}n}{2C_{\mathsf{var}}^{2}\sigma^{2}}\overset{\text{(v)}}{\leq}\frac{c^{2}}{2C_{\mathsf{var}}^{2}}k\nonumber \\
	& \leq\frac{8c^{2}}{C_{\mathsf{var}}^{2}}\log_{2}M\overset{\text{(vi)}}{\leq}\frac{1}{16}\log_{2}M<\frac{1}{8}\log M.\label{eq:proof-minimax-6}
\end{align}
Here (i) follows from the additivity property of the KL divergence
(cf.~\citet[page 85]{tsybakov2008introduction}) and for each $i\in[n]$,
$\theta_{i}^{\bm{\omega}}=(b')^{-1}(\mu_{i}^{\bm{\omega}})$; (ii)
follows from (\ref{eq:proof-minimax-5}); (iii) follows from the first
line of (\ref{eq:proof-minimax-4}); (iv) holds since the Hamming
distance between $k$-dimensional vectors is at most $k$; (v) follows
from (\ref{eq:proof-minimax-3.5}); and (vi) holds when $c\leq C_{\mathsf{var}}/16$.
Taking collectively (\ref{eq:proof-minimax-4}) and (\ref{eq:proof-minimax-6}),
we can utilize \citet[Theorem 2.7]{tsybakov2008introduction} to establish
the minimax lower bound (\ref{eq:proof-minimax-1}).

\paragraph{Case 2: when $n$ is small.}

We are left with the case when 
\begin{equation}
	\left(\frac{\widetilde{V}}{n\sigma}\right)^{2/3}\leq\frac{4}{n},\label{eq:proof-minimax-7}
\end{equation}
and we would like to show that 
\begin{equation}
	\inf_{\widehat{\bm{\mu}}}\sup_{\bm{\mu}^{\star}\in\mathcal{I}_{n}\left(V_{\min},V_{\max}\right)}\mathbb{E}\,\big\Vert\widehat{\bm{\mu}}-\bm{\mu}^{\star}\big\Vert^{2}\geq C_{2}\sigma^{2}.\label{eq:proof-minimax-8}
\end{equation}
We can consider the special case when all $\mu_{i}^{\star}$ are identical
and this information is known a priori. Then the problem reduces to
estimating $\mu^{\star}\in[V_{\min},V_{\max}]$ from i.i.d.~data
$X_{1},\ldots,X_{n}\sim p_{\theta^{\star}}(x)$
where $\theta^{\star}=(b')^{-1}(\mu^{\star})$, and it suffices to
show that 
\begin{equation}
	\inf_{\widehat{\mu}}\max_{\mu^{\star}\in\left[V_{\min},V_{\max}\right]}\big\vert\widehat{\mu}-\mu^{\star}\big\vert^{2}\geq C_{2}\frac{\sigma^{2}}{n}.\label{eq:proof-minimax-9}
\end{equation}
for some $C_{2}>0$. Let $\mu_{1}=\widetilde{V}_{\max}$ and $\mu_{2}=\widetilde{V}_{\max}-c\sigma/\sqrt{n}$
for some sufficiently small constant $c>0$. Since $\widetilde{V}_{\max}-\widetilde{V}_{\min}\geq C_{\mathsf{int}}(V_{\max}-V_{\min})$
and $V_{\max}-V_{\min}\geq\sigma/\sqrt{n}$, we know that $\mu_{2}\geq\widetilde{V}_{\min}$
as long as $c\leq C_{\mathsf{int}}$. Denote by $P_{\mu}$ the data
distribution when the sample average is $\mu$, and let $\theta_{1}=(b')^{-1}(\mu_{1})$
and $\theta_{2}=(b')^{-1}(\mu_{2})$. We have
\[
\mathsf{KL}\left(P_{\mu_{1}}\,\Vert\,P_{\mu_{2}}\right)\overset{\text{(i)}}{=}n\mathsf{KL}\left(p_{\theta_{1}}\,\Vert\,p_{\theta_{2}}\right)\overset{\text{(ii)}}{\leq}n\frac{\left(\mu_{1}-\mu_{2}\right)^{2}}{2C_{\mathsf{var}}^{2}\sigma^{2}}=\frac{c}{2C_{\mathsf{var}}^{2}}\overset{\text{(iii)}}{\leq}\frac{1}{8}.
\]
Here (i) follows from the additivity property of the KL divergence
(cf.~\citet[Page 85]{tsybakov2008introduction}); (ii) follows from
(\ref{eq:proof-minimax-5}); and (iii) holds as long as $c\leq C_{\mathsf{var}}^{2}/4$.
Then we can apply \citet[Theorem 2.2]{tsybakov2008introduction} to
obtain the desired minimax lower bound (\ref{eq:proof-minimax-9}),
which in turn establishes (\ref{eq:proof-minimax-8}).

Finally, taking (\ref{eq:proof-minimax-1}) and (\ref{eq:proof-minimax-8})
collectively gives the minimax lower bound (\ref{eq:general_minimax_lower})
as claimed.

%% file: discussion.tex
\section{Numerical experiments}

In this section, we conduct numerical experiments to illustrate the performance of the Isotonic Mechanism on a real data example from ICML 2023 and synthetic data.

\subsection{The ICML 2023 data}
\label{sec:icml-2023-data}

To test the Isotonic Mechanism, the second author of the present paper led an experiment at ICML 2023 through a collaboration between \url{openreview.net} and \url{openrank.cc}. The latter website was recently developed specifically for this experiment. Being the second largest conference in machine learning, ICML received a total of 6,538 submissions from 18,515 authors in 2023. Among 4,505 authors who submitted at least two papers, 1,331 authors provided the rankings of their papers for our experiment.

In evaluating the performance of the Isotonic Mechanism on the ICML 2023 ranking data, one challenge arises from the absence of ground-truth scores of the submissions, namely, ${\mu}_1^\star$, $\ldots$, ${\mu}_n^\star$. Recognizing that reviewers also provided their levels of confidence for their review scores, our approach is to use scores given by confident reviewers as a \textit{surrogate} for the ground-truth scores. More precisely, let a submission be reviewed by $m$ reviewers, with scores $X^{(1)}$, $\ldots$, $X^{(m)}$. Without loss of generality, suppose that the first $m - 1$ reviewers were more confident than the last reviewer.\footnote{In the case of ties, we pick the last reviewer from the least confident reviewers uniformly at random.} Formally, we treat the average from the $m-1$ relative confident scores,
\begin{equation}\label{eq:surrogate}
\widetilde{\mu}^\star = \frac{1}{m-1}\sum_{j=1}^{m-1} X^{(j)},
\end{equation}
as the surrogate for the ground-truth quality of this submission and use the least confident review score, $\widetilde X = X^{(m)}$, as the ``observed'' review score of the same submission. To facilitate data processing, we discarded 195 submissions that only had one review score. However, it is worth mentioning that, in practice, the average would be taken over all $m$ review scores instead of some $m-1$ review scores. As such, these 195 submissions could be included when implementing the Isotonic Mechanism in actual usage.

\begin{table}[]
	\centering
	\begin{tabular}{c|c|c|c|c|c|c|c|c}
		\hline
		\#Submissions                             & 2       & 3       & 4       & 5       & 6       & 7       & 8       & 9       \\ 
		Sample size                                  & 786     & 284     & 88      & 52      & 28      & 11      & 9       & 2       \\ \hline
		$\widehat{\mathsf{MSE}}_{\mathsf{raw}}$ & 2.0908  & 1.9414  & 2.4904  & 2.0741  & 1.9552  & 1.6054  & 1.5789  & 1.1605  \\ 
		$\widehat{\mathsf{MSE}}_{\mathsf{IM}}$  & 1.8290  & 1.7375  & 1.9842  & 1.6365  & 1.5720  & 1.2875  & 0.9502  & 0.9420  \\ 
		Improvement                                  & 12.52\% & 10.50\% & 20.33\% & 21.10\% & 19.60\% & 19.80\% & 39.82\% & 18.83\% \\ \hline\hline
		\#Submissions                             & 10       & 11       & 12       & 13     &14   & 15       & 16       & 17               \\ 
		Sample size                                  & 2     & 1     & 1      & 2   &0   & 1      & 2      & 2               \\ \hline
		$\widehat{\mathsf{MSE}}_{\mathsf{raw}}$ & 2.6139  & 2.0783  & 2.1991  & 1.8750 & NA & 1.2741  & 2.2474  & 1.5613     \\ 
		$\widehat{\mathsf{MSE}}_{\mathsf{IM}}$  & 1.3556  & 1.8294  & 1.3935  & 0.9850 & NA & 1.1574  & 1.7899  & 0.7040     \\
		Improvement                                  & 48.14\% & 11.98\% & 36.63\% & 47.47\% & NA & 9.16\% & 20.36\% & 54.91\%   \\ \hline
	\end{tabular}
	
	\caption{Experimental results on the ICML 2023 ranking data. The improvement of the Isotonic Mechanism over the original review scores is defined as $\big(\widehat{\mathsf{MSE}}_{\mathsf{raw}}-\widehat{\mathsf{MSE}}_{\mathsf{IM}}\big)/\widehat{\mathsf{MSE}}_{\mathsf{raw}}$. For each $2 \le n \le 17$, the comparison of the two approaches is evaluated for authors of $n$ submissions. }
\label{table:icml}
\end{table}

Consider an author $A$ who submitted $n$ papers with surrogate ground-truth $\widetilde{\mu}_1^\star$, $\ldots$, $\widetilde{\mu}_n^\star$ and noisy scores $\widetilde{X}_1$, $\ldots$, $\widetilde{X}_n$.\footnote{We excluded 60 authors in our experiment whose rankings were not informative in the sense that all papers of the same author were ranked first. } We run the Isotonic Mechanism using the ranking provided by this author to produce the adjusted scores $\widehat{\mu}_1$, $\ldots$, $\widehat{\mu}_n$. With the surrogate ground-truth in place, we evaluate the performance of the unadjusted review scores and the Isotonic Mechanism as follows:
\begin{equation}\label{eq:two-errors}
\widehat{\mathsf{MSE}}_{\mathsf{raw},A} = \frac{1}{n}\sum_{j=1}^{n} \left(\widetilde{X}_j-\widetilde{\mu}_j^\star\right)^2 \quad\text{and}\quad 
\widehat{\mathsf{MSE}}_{\mathsf{IM},A} = \frac{1}{n}\sum_{j=1}^{n} \left(\widehat{\mu}_j-\widetilde{\mu}_j^\star\right)^2.
\end{equation}
Next, we calculate the average across all authors with the same number of submissions for the two aforementioned errors:
\[
\widehat{\mathsf{MSE}}_{\mathsf{raw},n} = \frac{1}{|\mathcal{A}_n|} \sum_{A \in\mathcal{A}_n} \widehat{\mathsf{MSE}}_{\mathsf{raw},A} \quad\text{and}\quad 
\widehat{\mathsf{MSE}}_{\mathsf{IM},n} = \frac{1}{|\mathcal{A}_n|} \sum_{A \in\mathcal{A}_n} \widehat{\mathsf{MSE}}_{\mathsf{IM}, A},
\]
where $\mathcal{A}_n$ is the set of all authors who submitted $n$ papers. The results are presented in Table~\ref{table:icml}, which shows that the Isotonic Mechanism outperforms the original review scores consistently. In general, the gain becomes more significant as the number of submissions increases. In particular, it is nearly 20\% or more for authors with more than three submissions.

\subsection{Synthetic data}

In the analysis of the ICML 2023 data in Section~\ref{sec:icml-2023-data}, there is a mismatch that reduces the effectiveness of the Isotonic Mechanism. The error criterion $\widehat{\mathsf{MSE}}_{\mathsf{IM},A}$ utilizes the surrogate ground-truth $\widetilde{\mu}_1^\star$, $\ldots$, $\widetilde{\mu}_n^\star$, whereas the Isotonic Mechanism uses the author-provided ranking, which is presumably about ${\mu}_1^\star$, $\ldots$, ${\mu}_n^\star$. In real-world applications, however, this mismatch would not occur. Thus, in practical scenarios, the improvement gained from the Isotonic Mechanism might be even more significant than what is observed in Section~\ref{sec:icml-2023-data}.

To make a comparison from this viewpoint, we resort to synthetic data that mimic the ICML 2023 data. Let $\mathcal{X}_{\textnormal{ICML2023}}$ denote the collection of average review scores of all 6,345 submissions to ICML 2023.\footnote{Among the total 6,538 submission to ICML 2023, 193 were desk rejected or withdrawn before review scores were posted.} For any $2 \le n \le 17$, we generate a synthetic review instance as follows:
\begin{itemize}
	\item An author submits $n$ papers with ground-truth quality $\mu_1^\star,\ldots,\mu_n^\star$ that are sampled with replacement from $\mathcal{X}_{\textnormal{ICML2023}}$. 
	\item For each $1 \le i \le n$, let the average review score $X_i$ of the $i$th paper be the average of 3 i.i.d.~$\mathsf{Binomial}(10,\mu_i^\star/10)$ random variables.\footnote{Here, the number 3 was chosen because it is the median number of reviewers assigned to an ICML 2023 submission.}
	\item Run isotonic regression to obtain $\widehat{\mu}_1,\ldots,\widehat{\mu}_n$ for the synthetic review scores $X_1,\ldots,X_n$ and the ranking of $\mu_1^\star,\ldots,\mu_n^\star$. Calculate $\widehat{\mathsf{MSE}}_{\mathsf{raw}}$ and $\widehat{\mathsf{MSE}}_{\mathsf{IM}}$ as \eqref{eq:two-errors}, with $\widetilde{X}_j$ and  $\widetilde{\mu}_j^\star$ replaced by $X_j$ and ${\mu}_j^\star$, respectively.

\end{itemize}

\begin{table}[]
	\centering
	\begin{tabular}{c|c|c|c|c|c|c|c|c}
		\hline
		\#Submissions                           & 2       & 3       & 4       & 5       & 6       & 7       & 8       & 9       \\ \hline
		$\mathsf{Mean}(\widehat{\mathsf{MSE}}_{\mathsf{IM}})$ & 0.6408  &  0.6180  &  0.5619 &   0.4968  & 0.4558  &  0.4384  &  0.3942  &  0.3853  \\ 
		$\mathsf{Std}(\widehat{\mathsf{MSE}}_{\mathsf{IM}})$ & 0.6490 &   0.5172  &  0.4285 &   0.3561   & 0.3187  &  0.2825  &  0.2382 &   0.2379  \\ \hline
		$\mathsf{Mean}(\widehat{\mathsf{MSE}}_{\mathsf{raw}})$ & 0.7467  &  0.8005  &  0.8339  &  0.7911  &  0.7918  &  0.8195  &  0.7749  &  0.7884  \\ 
		$\mathsf{Std}(\widehat{\mathsf{MSE}}_{\mathsf{raw}})$ & 0.7241  &  0.6362  &  0.5739  &  0.5038   & 0.4570   & 0.4351  &  0.3651  &  0.3668  \\ 
		Improvement                                  & 14.18\% &  22.80\% &  32.61\% &   37.20\% &   42.43\% &   46.50\% &   49.12\% &   51.13\%  \\ \hline \hline
		\#Submissions                           & 10       & 11       & 12       & 13       & 14       & 15       & 16       & 17       \\ \hline
		$\mathsf{Mean}(\widehat{\mathsf{MSE}}_{\mathsf{IM}})$ & 0.3639  &  0.3406  &  0.3230  &  0.3176  &  0.2932  &  0.2910  &  0.2736  &  0.2677  \\ 
		$\mathsf{Std}(\widehat{\mathsf{MSE}}_{\mathsf{IM}})$ & 0.2245  &  0.1829  &  0.1756  &  0.1698   & 0.1507  &  0.1496  &  0.1339   & 0.1338  \\ \hline
		$\mathsf{Mean}(\widehat{\mathsf{MSE}}_{\mathsf{raw}})$ & 0.7925  &  0.7932  &  0.7934  &  0.7996  &  0.7900 &   0.7899  &  0.7884  &  0.8001  \\ 
		$\mathsf{Std}(\widehat{\mathsf{MSE}}_{\mathsf{raw}})$ & 0.3724  &  0.3329  &  0.3183  &  0.3007  &  0.2804 &   0.2886  &  0.2751  &  0.2738  \\ 
		Improvement                                  & 54.09\% &   57.06\% &   59.29\% &   60.29\% &   62.89\% &   63.16\% &   65.30\% &   66.55\%   \\ \hline
	\end{tabular}
	
	\caption{Experimental results on the synthetic data. The mean and standard deviation of the estimation errors are computed from 1,000 independent runs for each number of submissions.}
	\label{table:synthetic}
\end{table}

We run 1,000 independent trials for each $n$ to compute the mean and the standard deviation of the two estimation errors. Our results, shown in Table \ref{table:synthetic}, indeed demonstrates more significant gain of the Isotonic Mechanism than Table~\ref{table:icml} does. Even for a small number of submissions such as $n = 2$ or $3$, the gain of using the Isotonic Mechanism is already noticeable.

\subsection{More simulations on nonconvex utility functions}


\begin{figure}[t]
	\centering
	
	\begin{tabular}{c}
		\includegraphics[scale=0.6]{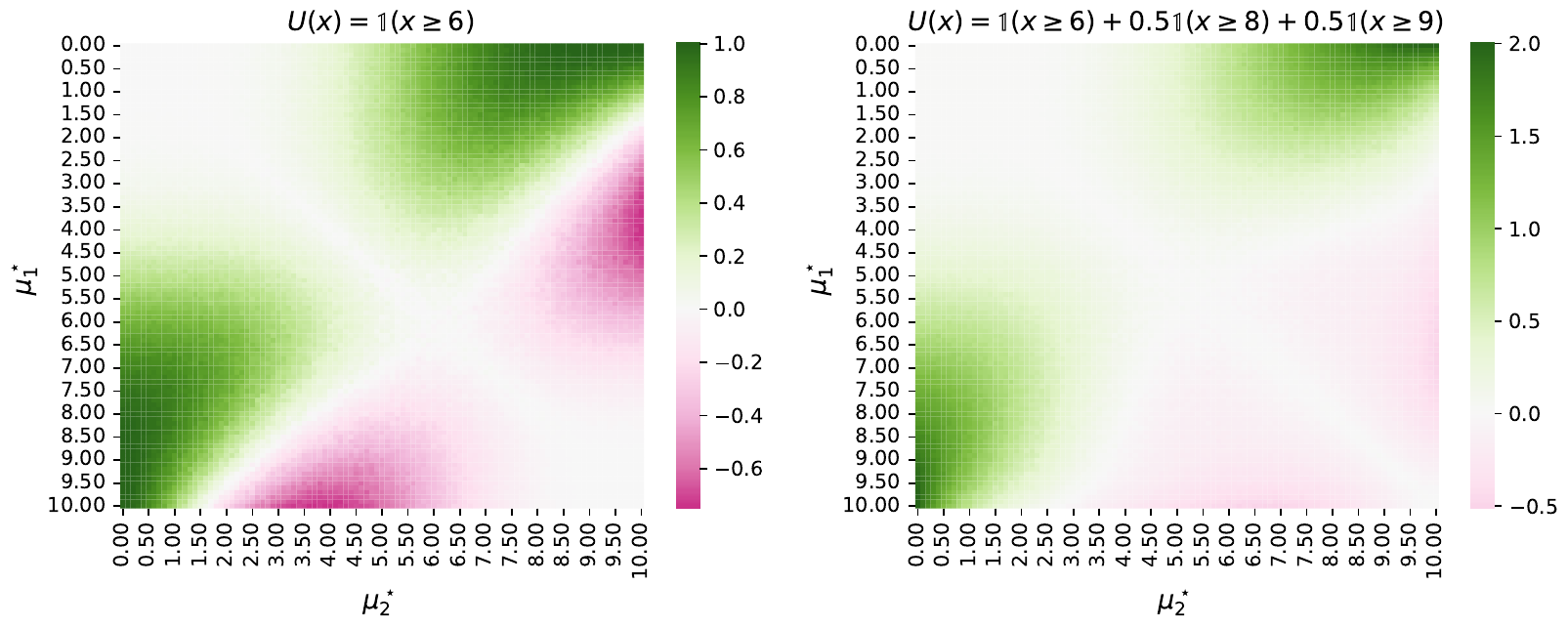}
	\end{tabular}
	
	\caption{The difference of utility between reporting the correct and incorrect ranking, for two different types of utility functions. The difference is positive (resp.~negative) means that the author has higher utility when reporting the correct (resp.~incorrect) ranking, and is represented by the green (resp.~red) color. \label{fig:heatmap}}
\end{figure}

In this section, we design numerical experiments to study whether the theory established in this paper still hold beyond the current model assumptions. Our truthfulness result (cf.~Theorem~\ref{thm:truth-telling}) relies on the convex utility assumption (cf.~Assumption~\ref{ass:convex}). While this assumption is satisfied if the author hopes to achieve a higher score to increase the chance of winning oral/spotlight presentation or best paper awards, it does not applies to the case where an author only cares about the acceptance of the paper. In this case, the utility of such an author can be given by an indicator function $U(x) = \ind \{ x \geq c\}$ for some threshold $c>0$. For example, when the review scores are between 1 and 10, taking $c=6$ reflects the utility that the author's utility only depends on the acceptance of the paper. A natural question is, under such a nonconvex utility function, whether it is possible for the author to get a higher overall utility by dishonestly manipulating the reported ranking. 

To answer this question, we consider the setting when an author submits two papers with true scores $\mu_1^\star,\mu_2^\star \in [0,10]$, and they receive noisy review scores $X_i \sim \mathsf{Bernoulli}(10,\mu_i^\star/10)$ for $i=1,2$ respectively. In this case, there are only two possible rankings(the ground truth ranking and the reversed ranking). We use Monte Carlo methods to evaluate the difference of the overall utility of reporting the true and reversed rankings for all possible pairs $(\mu_1^\star,\mu_2^\star) \in [0,10]^2$, namely
\[
\EE \left[ U( \widehat \mu_1 ) + U( \widehat \mu_2 ) \right] - \EE \left[ U(  \bar{\mu}_1 ) + U( \bar \mu_2 )\right], 
\] 
where $(\widehat{\mu}_1,\widehat{\mu}_2)$ (resp.~$(\bar{\mu}_1,\bar{\mu}_2)$) are the output of the isotonic mechanism with ground truth (resp.~reversed) ranking. The result for the utility function $U(x) = \ind\{x \geq 6\}$ is presented in the left panel of Figure~\ref{fig:heatmap}, which shows that even if one only cares the acceptance of the papers, for most $(\mu_1^\star,\mu_2^\star)$ pairs, the utility of reporting truthfully is still higher, with the exception when the difference between $\mu_1^\star$ and $\mu_2^\star$ is very large. Namely, it is possible to sacrifice the score of a very good paper a bit to gain a better chance for a much worse paper to get accepted. However we believe that this can hardly happen in practice, because of the following reasons.
\begin{itemize}
	\item It is reasonable to believe that an author capable of writing a very good paper (e.g., with true score $8$ -- $9$) will prefer to seek for higher chance of winning paper awards or oral presentation with such a good paper, rather than using it to save a low-quality paper.
	\item Given the review scores are highly noisy, it is actually risky to use this strategy in practice. Although the expected utility might be higher, there is a significant chance that both papers can end up being rejected, and we believe that most authors  may not want to take this risk.
	\item Most ML papers are written by several co-authors, who will submit their own rankings. It is not likely that they all wish to save a low-quality paper using the good one they co-authored. If they do not reach such a consensus and only one co-author is trying to manipulate his or her rank, we believe it is not difficult to detect this outlier from other co-authors' ranking.
\end{itemize}

Inspired by the above discussion, we also consider another nonconvex utility function
\[
U(x) = \ind\{x \geq 6\} + 0.5\ind\{x\geq 8\}+0.5\ind\{x\geq 9\},
\]
which also incorporates the utility from winning paper awards or oral presentations. The right panel of Figure~\ref{fig:heatmap} depicts the result for this utility function, which suggests that there is almost no gain in increasing the utility function by manipulating the ranking.

\section{Discussion}

In this paper, we examine the problem of eliciting private information from an owner to estimate a score vector generated from an exponential family distribution with varying parameters. Our primary contribution demonstrates that the Isotonic Mechanism \citep{su2021you, su2022truthful} maintains truthfulness in this extended regime, as the owner would optimize her convex utility by providing the ground-truth ranking for estimation. Additionally, we establish that the mechanism's truthfulness necessitates pairwise comparisons between the score means under certain conditions, which directly illustrates the optimality of ranking in truthful private information elicitation. Our analysis reveals that the Isotonic Mechanism achieves near-optimal estimation properties from a minimax perspective in a regime of bounded total variation, with the estimator produced by the mechanism significantly outperforming the original scores. This extension is made possible through the introduction of several new technical elements for analyzing the Isotonic Mechanism.

The extended Isotonic Mechanism is particularly well-suited for applications such as peer review in large machine learning and artificial intelligence conferences (e.g., ICML and NeurIPS), where review scores are integer-valued\footnote{In our setting, we consider the average score over three or four review scores, resulting in an integer multiple of $\frac13$ or $\frac14$.} and exhibit a strong dependence between mean and variance. The flexibility and richness of exponential family distributions aptly address this dependence. Furthermore, a key feature of the extended mechanism is its independence from the specific form of the exponential family distribution that generates the observations, rendering it easy to implement in practice.

Several important directions for future research emerge from our current work. One immediate question concerns further investigation of the necessary conditions for truthfulness by examining whether the intercepts in Theorem~\ref{thm:necessary-condition-slope} are equal to zero beyond the Gaussian case of Proposition~\ref{thm:necessary-condition-intercept}. Additionally, extending our mechanism to accommodate multivariate scores would be a welcome advance, which would allow for the evaluation of distinct aspects of a paper (e.g., novelty and writing quality). This multidimensional approach could provide a more nuanced and comprehensive evaluation framework. Another promising direction involves developing a ranking method for exponential family distributions in the context of multiple co-authors~\citep{wu2023isotonic}. From an empirical standpoint, assessing the robustness of the Isotonic Mechanism in scenarios where the author might have uncertainty in ranking the submissions would be valuable. This extension would explore situations where the author-provided ranking is only approximately truthful. More broadly, for journal submissions where authors typically submit one paper at a time, a ``credit'' system incorporating submission history with self-ratings could extend the Isotonic Mechanism to the sequential setting. Such a system could balance short-term evaluations with long-term reputational effects, presenting a worthwhile avenue for future research in peer review mechanisms.


%% file: appendix.tex
\section{Appendix}
\label{sec:appendix}

\subsection{Proof of Lemma \ref{lemma:upward-swap}\label{subsec:proof-lemma-upward-swap}}

We prove the lemma by induction. When $n=1$, the result is trivially
true. Suppose that the result holds for all $1\leq n\leq N-1$. Now
we show that the result also holds for $n=N$. Let $j=\pi^{-1}(N)$,
namely $j$ is the index such that $\pi(j)=N$. We define $\widetilde{\pi}$
by swapping the value of $\pi(j)$ and $\pi(N)$ in $\pi$:
\[
\widetilde{\pi}\left(i\right)=\begin{cases}
	N & \text{if }i=N,\\
	\pi\left(N\right) & \text{if }i=j,\\
	\pi\left(i\right) & \text{otherwise}.
\end{cases}
\]
It is straightforward to check that $\widetilde{\pi}$ is an upward
swap of $\pi$ when $j<N$, and $\widetilde{\pi}=\pi$ when $j=N$.
Since $\widetilde{\pi}(N)=N$ and $\pi^{\star}(N)=N$, by restricting
$\pi^{\star}$ and $\widetilde{\pi}$ to $[N-1]$, we obtain two permutations
$\overline{\pi}^{\star}$ and $\overline{\pi}$ over $[N-1]$, namely
\[
\overline{\pi}^{\star}\left(i\right)\coloneqq\pi^{\star}\left(i\right),\qquad\overline{\pi}\left(i\right)=\widetilde{\pi}\left(i\right),\qquad\forall\,i\in\left[N-1\right].
\]
Then we can apply the induction hypothesis to show that there exist
$m\geq1$ and a sequence of permutations over $[N-1]$: $\overline{\pi}_{1},\ldots,\overline{\pi}_{m}$
such that $\overline{\pi}_{1}=\overline{\pi}^{\star}$, $\overline{\pi}_{m}=\overline{\pi}$,
and for all $1\leq i\leq m-1$, $\overline{\pi}_{i}$ is an upward
swap of $\overline{\pi}_{i+1}$. Then we extend $\overline{\pi}_{1},\ldots,\overline{\pi}_{m}$
to be permutations $\widetilde{\pi}_{1}$, ..., $\widetilde{\pi}_{m}$
over $[N]$ by defining
\[
\widetilde{\pi}_{k}\left(i\right)=\begin{cases}
	\overline{\pi}_{k}\left(i\right) & \text{if }i\leq N-1,\\
	N & \text{if }i=N,
\end{cases}
\]
for all $k\in[m]$ and $i\in[N]$. When $j<N$, let 
\[
\pi_{1}=\widetilde{\pi}_{1},\quad\pi_{2}=\widetilde{\pi}_{2},\quad\ldots\quad\pi_{m}=\widetilde{\pi}_{m},\quad\pi_{m+1}=\pi.
\]
It is straightforward to check that $\pi_{1}=\pi^{\star}$, $\pi_{m}=\pi$,
and for each $1\leq i\leq m$, $\pi_{i}$ is an upward swap of $\pi_{i+1}$.
When $j=N$, let 
\[
\pi_{1}=\widetilde{\pi}_{1},\quad\pi_{2}=\widetilde{\pi}_{2},\quad\ldots\quad\pi_{m}=\widetilde{\pi}_{m}.
\]
We can also check that $\pi_{1}=\pi^{\star}$, $\pi_{m}=\pi$, and
for each $1\leq i\leq m-1$, $\pi_{i}$ is an upward swap of $\pi_{i+1}$.
This concludes the proof.
